\documentclass{article}

\usepackage{tabularx}
\usepackage{graphicx}
\usepackage{fullpage}
\usepackage{hyperref}
\usepackage{natbib}
\usepackage{delarray}
\usepackage{times}
\usepackage{subfigure}
\usepackage{color}
\usepackage{amssymb}
\usepackage{amsmath}
\usepackage{tikz}
\usepackage{xspace}
\usepackage{enumerate}
\usepackage{algorithm}
\usepackage{algorithmic}
\usepackage{multirow}
\usepackage{nameref}
\usepackage{booktabs}
\usepackage{mdframed}
\usepackage{fmtcount}
\usepackage{footnote}
\makesavenoteenv{table}
\makesavenoteenv{tabular}
\makesavenoteenv{algorithm}
\makesavenoteenv{algorithmic}
\makesavenoteenv{figure}

\newcommand{\eat}[1]{}
\newcommand{\topic}[1]{\vspace{5pt}\noindent{{\bf #1:}}}

\newcommand{\R}{{\mathbb R}}

\newcommand{\E}{{\mathbb{E}}}

\newcommand{\calG}{{\mathcal G}}
\newcommand{\prox}{\mathrm{prox}}
\newcommand{\red}[1]{\textcolor{red}{#1}}
\newcommand{\blue}[1]{\textcolor{blue}{#1}}

\newcommand{\inner}[2]{\langle #1, #2 \rangle}

\newcommand{\ns}[1]{\| #1 \|^2}
\newcommand{\nl}[1]{\| #1 \|_1}
\newcommand{\n}[1]{\| #1 \|}
\newcommand{\hx}{\hat{x}}

\newcommand{\tx}{\widetilde{x}}

\newcommand{\bx}{\bar{x}}

\newenvironment{proof}{\noindent {\em Proof: }\ignorespaces}%
                {\hspace*{\fill}$\Box$\par}
        {\hspace*{\fill}$\Box$\par\vspace{4mm}}
\newenvironment{proofof}[1]{\smallskip\noindent{\bf Proof of #1.}}%
        {\hspace*{\fill}$\Box$\par}

\newtheorem{theorem}{Theorem}
\newtheorem{lemma}{Lemma}

\newtheorem{definition}{Definition}
\newtheorem{assumption}{Assumption}

\hypersetup{
    colorlinks=true,       
    linkcolor=black,    
    citecolor=black,        
    filecolor=magenta,     
    urlcolor=blue
}

\begin{document}

\title{A Simple Proximal Stochastic Gradient Method for Nonsmooth Nonconvex Optimization\footnote{32nd Conference on Neural Information Processing Systems
      (NeurIPS 2018), Montr\'{e}al, Canada}}
\author{Zhize Li  \qquad \qquad \qquad
        Jian Li \\
      Institute for Interdisciplinary Information Sciences, Tsinghua University
      \\ zz-li14@mails.tsinghua.edu.cn \quad
        lijian83@mail.tsinghua.edu.cn}

\date{}
\maketitle

\begin{abstract}
We analyze stochastic gradient algorithms for optimizing nonconvex, nonsmooth finite-sum problems. In particular, the objective function is given by the summation of a differentiable (possibly nonconvex) component, together with a possibly non-differentiable but convex component.
We propose a proximal stochastic gradient algorithm based on variance reduction, called ProxSVRG+.
Our main contribution lies in the analysis of ProxSVRG+.
It recovers several existing convergence results and improves/generalizes them (in terms of the number of stochastic gradient oracle calls and proximal oracle calls).
In particular, ProxSVRG+ generalizes the best results given by the SCSG algorithm, recently proposed by \citep{lei2017non} for the smooth nonconvex case.
ProxSVRG+ is also more straightforward than SCSG and yields simpler analysis.
Moreover, ProxSVRG+ outperforms the deterministic proximal gradient descent (ProxGD) for a wide range of minibatch sizes, which partially solves an open problem proposed in \citep{reddi2016proximal}.
Also, ProxSVRG+ uses much less proximal oracle calls than ProxSVRG \citep{reddi2016proximal}.
Moreover, for nonconvex functions satisfied Polyak-\L{}ojasiewicz condition, we prove that ProxSVRG+ achieves a global linear convergence rate without restart unlike ProxSVRG.
Thus, it can \emph{automatically} switch to the faster linear convergence in some regions as long as the objective function satisfies the PL condition locally in these regions.
ProxSVRG+ also improves ProxGD and ProxSVRG/SAGA, and generalizes the results of SCSG in this case.
Finally, we conduct several experiments and the experimental results are consistent with the theoretical results.
\end{abstract}

\section{Introduction}
\label{sec:intro}
In this paper, we consider nonsmooth nonconvex finite-sum optimization problems of the form
\begin{equation}\label{eq:form}
\min_{x}\Phi(x) := f(x)+h(x),
\end{equation}
where $f(x) := \frac{1}{n}\sum_{i=1}^n{f_i(x)}$ and each $f_i(x)$ is possibly nonconvex with a Lipschitz continuous gradient,
while $h(x)$ is nonsmooth but convex (e.g., $l_1$ norm $\nl{x}$ or indicator function $I_C(x)$ for some convex set $C$).
We assume that the proximal operator of $h(x)$ can be computed efficiently.

This above optimization problem
is fundamental to many machine learning problems, ranging from
convex optimization such as Lasso, SVM to highly nonconvex problem such as optimizing deep neural networks.
There has been extensive research when $f(x)$ is convex
(see e.g.,~\citep{xiao2014proximal,defazio2014saga,lan2015optimal,allen2017katyusha}).
In particular, if $f_i$s are strongly-convex, \cite{xiao2014proximal} proposed the Prox-SVRG algorithm, which achieves
a linear convergence rate,
based on the well-known variance reduction technique SVRG developed in
\citep{johnson2013accelerating}.
In recent years,
due to the increasing popularity of deep learning,
the nonconvex case has attracted significant attention.
See e.g.,~\citep{ghadimi2013stochastic,allen2016variance,reddi2016stochastic,lei2017non}
for results on
the smooth nonconvex case (i.e., $h(x)\equiv 0$).
Very recently, \cite{zhou2018stochastic} proposed an algorithm with stochastic gradient complexity $\widetilde{O}(\frac{1}{\epsilon^{3/2}}\wedge \frac{n^{1/2}}{\epsilon})$, improving the previous results $O(\frac{1}{\epsilon^{5/3}})$ \citep{lei2017non} and $O(\frac{n^{2/3}}{\epsilon})$ \citep{allen2016variance}.
For the more general \emph{nonsmooth nonconvex} case,
the research is still somewhat limited.

Recently, for the nonsmooth nonconvex case, \cite{reddi2016proximal} provided two
algorithms called ProxSVRG and ProxSAGA, which are based on the well-known variance reduction techniques SVRG and SAGA \citep{johnson2013accelerating,defazio2014saga}.
Also, we would like to mention that \cite{aravkin2016smart}
considered the case when $h$ can be nonconvex in a more general context of robust optimization.
Before that, \cite{ghadimi2016mini} analyzed the deterministic proximal gradient method (i.e., computing the full-gradient in every iteration) for nonconvex nonsmooth problems. Here we denote it as ProxGD.
\cite{ghadimi2016mini} also considered the stochastic case (here we denote it as ProxSGD).
However, ProxSGD requires the batch sizes being a large number (i.e., $\Omega(1/\epsilon)$) or increasing with the iteration number $t$.
Note that ProxSGD may reduce to deterministic ProxGD after some iterations
due to the increasing batch sizes.
Note that from the perspectives of both computational efficiency and
statistical generalization, always computing full-gradient (GD or ProxGD) may not be desirable for large-scale machine learning problems.
A reasonable minibatch size is also desirable in practice, since the computation of minibatch stochastic gradients can be implemented in parallel.
In fact, practitioners typically use moderate minibatch sizes,
often ranging from something like 16 or 32 to a few hundreds (sometimes to a few thousands, see e.g., \citep{goyal2017accurate}).\footnote{In fact, some studies argued that smaller minibatch sizes
	in SGD are very useful for generalization
	(e.g., \citep{keskar2016large}).
	Although generalization is not the focus of the present paper,
	it provides further motivation for studying the moderate
	minibatch size regime.
}
Hence, it is important to study the convergence
in moderate and constant minibatch size regime.

\cite{reddi2016proximal} provided the first non-asymptotic convergence rates for
ProxSVRG with minibatch size at most $O(n^{2/3})$, for the nonsmooth nonconvex problems.
However, their convergence bounds (using constant or moderate size minibatches)
are worse than the deterministic ProxGD in terms of the number of proximal oracle calls.
Note that their algorithms (i.e., ProxSVRG/SAGA) outperform the ProxGD only if they use quite large minibatch size $b=O(n^{2/3})$.
Note that in a typical application, the number of training data is $n=10^6\sim 10^9$, and $n^{2/3}=10^4\sim 10^6$.
Hence, $O(n^{2/3})$ is a quite large minibatch size.
Finally, they presented an important open problem of
\emph{developing stochastic methods with provably better performance than ProxGD with constant minibatch size}.

\vspace{-1mm}
\topic{Our Contribution} In this paper, we propose a very straightforward algorithm called ProxSVRG+ to solve the nonsmooth nonconvex problem \eqref{eq:form}.
Our main technical contribution lies in the new convergence
analysis of ProxSVRG+, which has notable difference from that of ProxSVRG \citep{reddi2016proximal}.
We list our results in Table \ref{tab:1}--\ref{tab:3} and Figure \ref{fig:1}--\ref{fig:2}.
Our convergence results are stated in terms of the number of
stochastic first-order oracle (\emph{SFO}) calls and proximal oracle (\emph{PO}) calls (see Definition \ref{dfn:oracle}).
We would like to highlight the following results
yielded by our new analysis:

\vspace{-1.5mm}
\begin{enumerate}[1)]
  \item
      ProxSVRG+ is $\sqrt{b}$ (resp. $\sqrt{b}\epsilon n$) times faster than
      ProxGD in terms of \#SFO when $b\leq n^{2/3}$
      (resp. $b\leq 1/\epsilon^{2/3}$),
      and $n/b$ times faster than
      ProxGD when $b>n^{2/3}$ (resp. $b>1/\epsilon^{2/3}$).
      Note that \#PO $=O(1/\epsilon)$ for both ProxSVRG+ and ProxGD.
      Obviously, for any super constant $b$,
      ProxSVRG+ is strictly better than ProxGD.
      Hence, we partially answer the open question (i.e. developing stochastic methods with provably better performance than ProxGD with constant minibatch size $b$) proposed in
      \citep{reddi2016proximal}.
      ProxSVRG+ also matches the
      best result achieved by ProxSVRG at $b=n^{2/3}$,
      and it is strictly better for smaller $b$ (using less
      PO calls). See Figure \ref{fig:1} for an overview.

  \item
      Assuming that the variance of the stochastic gradient is bounded (see Assumption \ref{asp:var}), i.e. online/stochastic setting,
      ProvSVRG+ generalizes the best result achieved by SCSG, recently proposed by \citep{lei2017non} for the smooth nonconvex case, i.e., $h(x)\equiv 0$ in form (\ref{eq:form})
      (see Table \ref{tab:1}, the 5th row).
      ProxSVRG+ is more straightforward than SCSG
      and yields simpler proof.
      Our results also match the results of
      Natasha1.5 proposed by \citep{allen2017natasha} very recently,
      in terms of \#SFO, if there is no additional assumption
      (see Footnote \ref{foot:natasha} for details).
      In terms of \#PO, our algorithm outperforms Natasha1.5.

      We also note that SCSG \citep{lei2017non} and ProxSVRG \citep{reddi2016proximal} achieved their best convergence results with $b=1$ and $b=n^{2/3}$ respectively,
      while ProxSVRG+ achieves the best result with $b=1/\epsilon^{2/3}$ (see Figure \ref{fig:1}), which is a moderate minibatch size
      (which is not too small for parallelism/vectorization and not too large for better generalization).
      In our experiments, the best $b$ for ProxSVRG and ProxSVRG+ in the MNIST experiments is 4096 and 256, respectively (see the second row of Figure \ref{fig:n4}).

  \item
      For the nonconvex functions satisfying	
      Polyak-\L{}ojasiewicz condition \citep{polyak1963gradient},
      we prove that ProxSVRG+ achieves a global linear convergence rate \emph{without restart},
      while \cite{reddi2016proximal} used PL-SVRG to restart ProxSVRG many times to obtain the linear convergence rate. Thus, ProxSVRG+ can \emph{automatically} switch to the faster linear convergence in some regions.
      ProxSVRG+ also improves ProxGD and ProxSVRG/SAGA, and generalizes the results of SCSG in this case (see Table \ref{tab:3}). Also see the remarks after Theorem \ref{thm:pl1} for more details.
\end{enumerate}

\begin{table}[!htb]
\vspace{-5mm}
\centering
\caption{Comparison of the SFO and PO complexity}
\vspace{1.5mm}
\label{tab:1}
    \begin{tabular}{|c|c|c|c|}
	\hline
	\multirow{2}*{Algorithms}  & Stochastic first-order & Proximal oracle   & Additional \\
			& oracle (SFO)         &  (PO)                   & condition\\ \hline
			
	ProxGD~\citep{ghadimi2016mini}    & \multirow{2}*{$O(n/\epsilon)$} & \multirow{2}*{$O(1/\epsilon)$}
		      & \multirow{2}*{--} \\
	 (full gradient) &&& \\\hline
			
	\multirow{2}*{ProxSGD~\citep{ghadimi2016mini}}
	& \multirow{2}*{$O(b/\epsilon)$}
	& \multirow{2}*{$O(1/\epsilon)$} &$\sigma =O(1),$ \\
	&&& $b\ge 1/\epsilon$\\\hline

	\multirow{2}*{ProxSVRG/SAGA~\citep{reddi2016proximal}} &
            \multirow{2}*{$O\big(\frac{n}{\epsilon\sqrt{b}}+n\big)$}
			& \multirow{2}*{$O\big(\frac{n}{\epsilon b^{3/2}}\big)$}
			& \multirow{2}*{$b\le n^{2/3}$} \\ &&&  \\\hline
			
    SCSG~\citep{lei2017non}
    &\multirow{3}*{$O\Big(\frac{b^{1/3}}{\epsilon}\big(n\wedge\frac{1}{\epsilon}\big)^{2/3}\Big)$}
			&\multirow{3}*{NA}& \multirow{3}*{$\sigma = O(1)$} \\
			(smooth nonconvex, &&&\\
			i.e., $h(x)\equiv 0$ in (\ref{eq:form})) &&& \\\hline
			
	Natasha1.5~\citep{allen2017natasha} & $O(1/\epsilon^{5/3})$\,\,\footnotemark
	 &  $O(1/\epsilon^{5/3})$ & $\sigma = O(1)$ \\\hline

	& \multirow{2}*{$O\big(\frac{n}{\epsilon\sqrt{b}}+\frac{b}{\epsilon}\big)$}
		& \multirow{2}*{$O(1/\epsilon)$}   & \multirow{2}*{--} \\
		\red{ProxSVRG+} &&& \\\cline{2-4}
		(this paper)
			&\multirow{2}*{$O\Big(\big(n\wedge\frac{1}{\epsilon}\big)\frac{1}{\epsilon\sqrt{b}}
				+\frac{b}{\epsilon}\Big)$}
			& \multirow{2}*{$O(1/\epsilon)$}  & \multirow{2}*{$\sigma = O(1)$} \\
			&&& \\\hline
	\end{tabular}
		
	\vspace{1mm}
	{\small The notation $\wedge$ denotes the minimum and $b$ denotes the minibatch size.
			The definitions of SFO and PO are given \\ in Definition \ref{dfn:oracle}, and $\sigma$ (in the last column) is defined in Assumption \ref{asp:var}.}
\end{table}
\footnotetext{Natasha 1.5 used an additional parameter,
called strongly nonconvex parameter $\widetilde{L}$ ($\widetilde{L}\leq L$)
and \#SFO in~\citep{allen2017natasha} is
$O(\frac{1}{\epsilon^{3/2}}+\frac{\widetilde{L}^{1/3}}{\epsilon^{5/3}})$.
If $\widetilde{L}$ is much smaller than $L$, the bound is better.
Without any additional assumption, the default value of $\widetilde{L}$ is $L$.
The result listed in the table is the $\widetilde{L}=L$ case.
Besides, one can verify that  \#PO of Natasha1.5 is the same as its \#SFO.\label{foot:natasha}}

\begin{table}[!htb]
\vspace{-7mm}
\centering
\caption{Some recommended minibatch sizes $b$}
\vspace{1.5mm}
\label{tab:2}
    \begin{tabular}{|c|c|c|c|c|c|}
	   \hline
	   Algorithm & Minibatches& SFO & PO  & Addi. cond. & Notes\\ \hline
			
		\multirow{11}*{\red{ProxSVRG+}}
			&\multirow{2}*{$b=1$}
			&$O(n/\epsilon)$ & $O(1/\epsilon)$ &-- & Same as ProxGD\\\cline{3-6}
			& & $O(1/\epsilon^2)$ & $O(1/\epsilon)$ & $\sigma = O(1)$ & Same as ProxSGD \\\cline{2-6}
			
			&\multirow{5}*{$b=\frac{1}{\epsilon^{2/3}}$}
			&\multirow{2}*{\red{$O\big(\frac{n}{\epsilon^{2/3}} + \frac{1}{\epsilon^{5/3}}\big)$}}
			& \multirow{2}*{\red{$O(1/\epsilon)$}}
			& \multirow{2}*{--} & \red{Better than ProxGD,} \\
			&&&&& \red{does not need $\sigma = O(1)$}\\\cline{3-6}
			
			& &\multirow{3}*{\blue{$O\big(\frac{1}{\epsilon^{5/3}}\big)$}}&
               \multirow{3}*{\blue{$O(1/\epsilon)$}}
			&&\blue{Better than ProxGD and}\\ &&&&$\sigma = O(1),$&\blue{ProxSVRG/SAGA,}\\
			&&&& \blue{$n>1/\epsilon$} &\blue{same as SCSG (in SFO)}\\\cline{2-6}
			
			&\multirow{2}*{$b=n^{2/3}$}
			&\multirow{2}*{$O\big(\frac{n^{2/3}}{\epsilon}\big)$} & \multirow{2}*{$O(1/\epsilon)$}
			& \multirow{2}*{--} &Same as\\ &&&& & ProxSVRG/SAGA\\\cline{2-6}
			
			&$b=n$ &$O(n/\epsilon)$ & $O(1/\epsilon)$ &--  &Same as ProxGD\\\hline
    \end{tabular}
\end{table}

\begin{figure}[!htb]
\vspace{-6.5mm}
\centering
	\begin{minipage}[htb]{0.5\textwidth}
		\centering
		\includegraphics[width=0.82\textwidth]{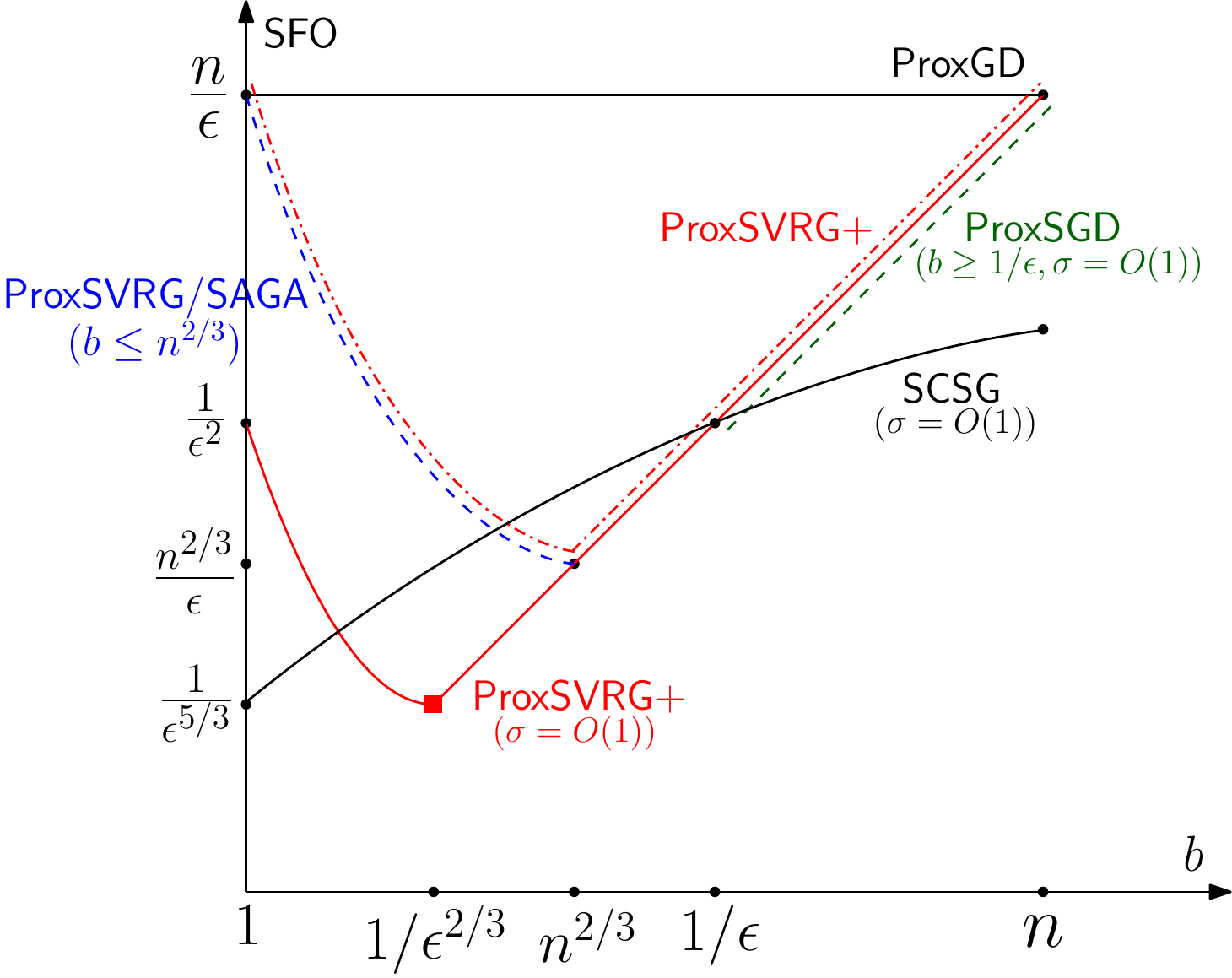}
		\vspace{-4mm}
		\caption{\small SFO complexity in terms of minibatch size $b^{\,\,3}$}
		\label{fig:1}
	\end{minipage}\hspace{5pt}
	\begin{minipage}[htb]{0.45\textwidth}
		\centering
		\vspace{1.5mm}
		\includegraphics[width=0.7\textwidth]{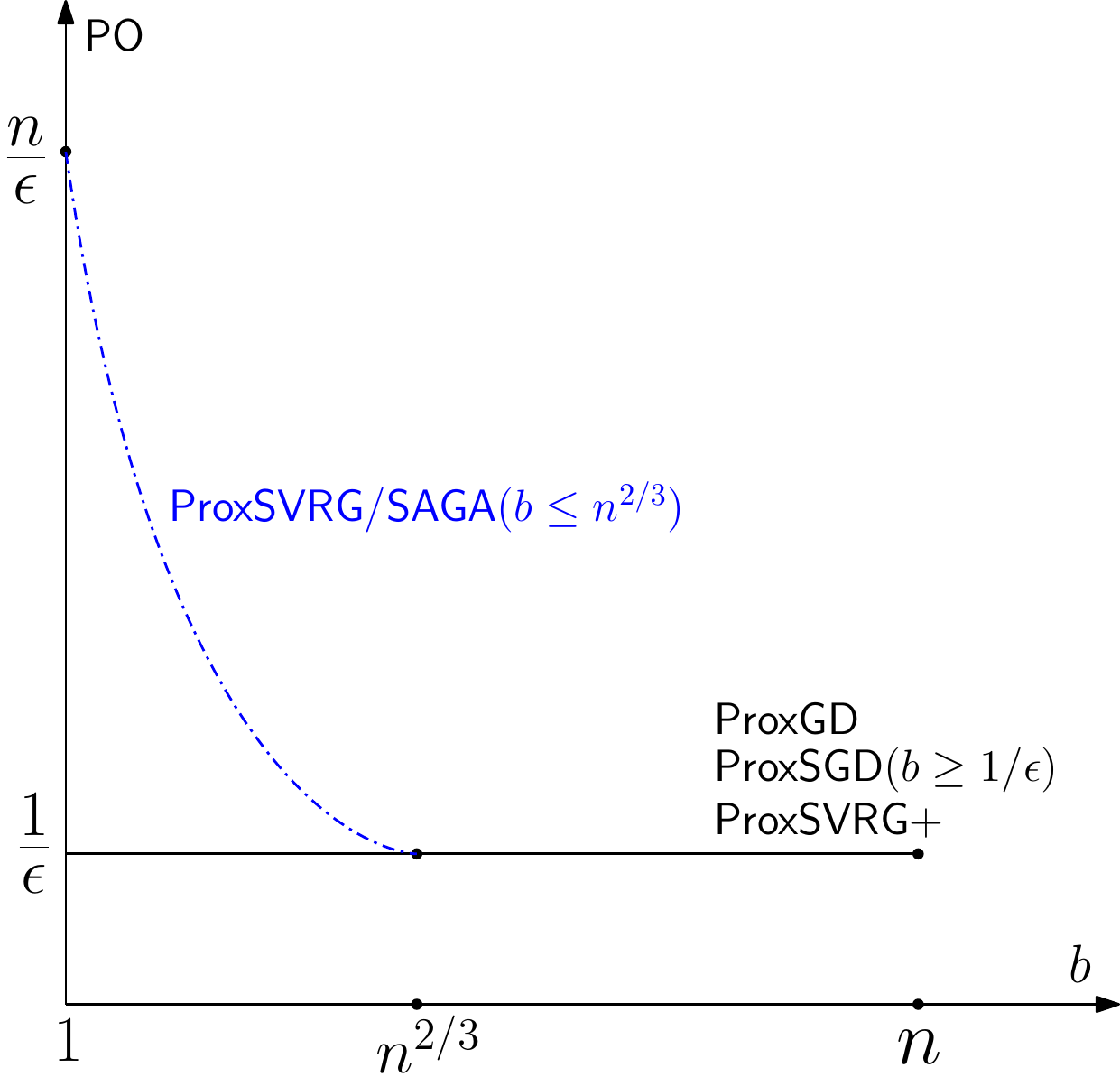}
		\vspace{-1.9mm}
		\caption{\small PO complexity in terms of minibatch size $b$}
		\label{fig:2}
	\end{minipage}
	\vspace{-3.5mm}
\end{figure}
\footnotetext[3]{Note that the curve of ProxSGD overlaps with ProxSVRG+ for $b\geq 1/\epsilon$, and the curve of ProxSVRG/SAGA overlaps with ProxSVRG+ for $b\leq n^{2/3}$ in Figure \ref{fig:1}.
We did not plot Natasha 1.5 since it did not consider the minibatch case, i.e., $b\equiv 1$ in Natasha 1.5.
}\setcounter{footnote}{3}

\section{Preliminaries}
\label{sec:pre}
We assume that $f_i(x)$ in (\ref{eq:form}) has an $L$-Lipschitz continuous gradient for all $i\in[n]$, i.e., there is a constant $L$ such that
\begin{equation}\label{eq:smooth}
\n{\nabla f_i(x)-\nabla f_i(y)}\leq L\n{x-y},
\end{equation}
where $\n{\cdot}$ denotes the Eculidean norm $\|\cdot\|_2$.
Note that $f_i(x)$ does not need to be convex.
We also assume that the nonsmooth convex function $h(x)$ in (\ref{eq:form}) is well structured, i.e.,
the following proximal operator on $h$ can be computed efficiently:
\begin{equation}\label{eq:prox}
	\prox_{\eta h}(x) := \arg\min_{y\in\R^d}{\Big(h(y)+\frac{1}{2\eta}\ns{y-x}\Big)}.
\end{equation}
For convex problems, one typically uses the optimality gap $\Phi(x)-\Phi(x^*)$ as the convergence criterion (see e.g., \citep{nesterov2014introductory}).
But for general nonconvex problems, one typically uses the gradient norm as the convergence criterion.
E.g., for smooth nonconvex problems (i.e., $h(x)\equiv 0$), \cite{ghadimi2013stochastic}, \cite{reddi2016stochastic} and \cite{lei2017non} used $\ns{\nabla\Phi(x)}$ (i.e., $\ns{\nabla f(x)}$) to measure the convergence results.
In order to analyze the convergence results for \emph{nonsmooth} nonconvex problems, we need to define the \emph{gradient mapping}
as follows (as in \citep{ghadimi2016mini,reddi2016proximal}):
\begin{equation}\label{eq:gradmap}
\calG_\eta(x) := \frac{1}{\eta}\Big(x - \prox_{\eta h}\big(x-\eta\nabla f(x)\big)\Big).
\end{equation}
We often use an equivalent but useful form of
$\prox_{\eta h}\big(x-\eta\nabla f(x)\big)$ as follows:
\begin{equation}\label{eq:prox2}
 \prox_{\eta h}\big(x-\eta\nabla f(x)\big) = \arg\min_{y\in\R^d}{\Big(h(y)+\frac{1}{2\eta}\ns{y-x}+\inner{\nabla f(x)}{y}\Big)}.
\end{equation}

Note that if $h(x)$ is a constant function (in particular, zero),
this gradient mapping reduces to  the ordinary gradient:
$\calG_\eta(x) = \nabla\Phi(x) = \nabla f(x)$.
In this paper, we use the gradient mapping $\calG_\eta(x)$
as the convergence criterion (same as \citep{ghadimi2016mini,reddi2016proximal}).
\begin{definition}
\label{dfn:eps}
$\hat{x}$ is called an $\epsilon$-accurate solution for problem (\ref{eq:form})
if $\E[\ns{\calG_\eta(\hx)}]\leq \epsilon$,
where $\hat{x}$ denotes the point returned by a stochastic algorithm.
\end{definition}
Note that the metric $\calG_\eta(x)$ has already normalized the step-size $\eta$, i.e., it is independent of different algorithms. Also it is indeed a convergence metric for $\Phi(x) = f(x)+h(x)$.
Let $x^+:=\prox_{\eta h}\big(x-\eta\nabla f(x)\big)$, then $\calG_\eta(x) := \frac{1}{\eta}\big(x - x^+\big)$.
If $\|\calG_\eta(x)\|=\frac{1}{\eta}\|x - x^+\|=\|\nabla f(x) + \partial h(x^+)\|\leq \epsilon$, then $\|\partial \Phi(x^+)\|=\|\nabla f(x^+) + \partial h(x^+)\|\leq L\|x-x^+\|+\|\nabla f(x) + \partial h(x^+)\|\leq L\eta\epsilon +\epsilon=O(\epsilon)$. Thus the next iteration point $x^+$ is an $\epsilon$-approximate stationary solution for the objection function $\Phi(x) = f(x)+h(x)$.

To measure the efficiency of a stochastic algorithm, we use the following oracle complexity.
\begin{definition}\label{dfn:oracle}
\begin{enumerate}[(1)]
  \item Stochastic first-order oracle (SFO): given a point $x$, SFO outputs a stochastic gradient $\nabla f_i(x)$ such that $\E_{i\sim [n]}[\nabla f_i(x)]=\nabla f(x)$.
  \item Proximal oracle (PO): given a point $x$, PO outputs
  the result of the proximal projection $\prox_{\eta h}(x)$ (see (\ref{eq:prox})).
\end{enumerate}
\end{definition}

Sometimes, the following assumption
on the variance of the stochastic gradients is needed
(see the last column ``additional condition" in Table \ref{tab:1}).
Such an assumption is necessary if one wants the convergence result
to be independent of $n$.
People also denote this case as the online/stochastic setting, in which the full gradient is not available (see e.g., \citep{allen2017natasha,lan2018random}).

\begin{assumption}\label{asp:var}
For $\forall x$, $\E[\ns{\nabla f_i(x)-\nabla f(x)}]\leq \sigma^2$,
where $\sigma>0$ is a constant and $\nabla f_i(x)$ is a stochastic gradient.
\end{assumption}

\section{Nonconvex ProxSVRG+ Algorithm}
\label{sec:alg}
In this section, we propose a proximal stochastic gradient algorithm called ProxSVRG+, which is very straightforward (similar to nonconvex ProxSVRG \citep{reddi2016proximal} and convex Prox-SVRG \citep{xiao2014proximal}).
The details are described in Algorithm \ref{alg:1}.
We call $B$ the batch size and $b$ the minibatch size.

\newpage
\begin{algorithm}[!t]
	\caption{Nonconvex ProxSVRG+}
	\label{alg:1}
	\begin{algorithmic}[1]
		\REQUIRE 
		initial point $x_0$, batch size $B$, minibatch size $b$, epoch length $m$, step size $\eta$

		\STATE $\tx^0 = x_0$
        \FOR{$s=1, 2, \ldots, S$}
        \STATE $x_0^s=\tx^{s-1}$
        \STATE
          $g^s=\frac{1}{B}\sum_{j\in I_B}\nabla f_j(\tx^{s-1})$\,\,\footnotemark
            \FOR{$t=1, 2, \ldots, m$}
            \STATE
            $v_{t-1}^s=\frac{1}{b}\sum_{i\in I_b}\big(\nabla f_i(x_{t-1}^s)-\nabla f_i(\tx^{s-1})\big) + g^s$
            \STATE $x_t^s = \prox_{\eta h}(x_{t-1}^s - \eta v_{t-1}^s)$ (call PO once)
            \ENDFOR
        \STATE $\tx^s = x_m^s$
        \ENDFOR
        \ENSURE $\hx$ chosen uniformly from $\{x_{t-1}^s\}_{t\in[m], s\in[S]}$
	\end{algorithmic}
\end{algorithm}
\footnotetext{If $B=n$, ProxSVRG+ is almost the same as ProxSVRG (i.e., $g^s=\frac{1}{n}\sum_{j=1}^n\nabla f_j(\tx^{s-1})=\nabla f(\tx^{s-1})$) except some detailed parameter settings (e.g., step-size, epoch length).}

Compared with Prox-SVRG, ProxSVRG \citep{reddi2016proximal} analyzed the nonconvex functions while Prox-SVRG \citep{xiao2014proximal} only analyzed the convex functions.
The major difference of our ProxSVRG+ is that we avoid the computation of the full gradient at the beginning of each epoch, i.e., $B$ may not equal to $n$ (see Line 4 of Algorithm \ref{alg:1}) while ProxSVRG and Prox-SVRG used $B=n$. Note that even if we choose $B=n$, our analysis is more stronger than ProxSVRG \citep{reddi2016proximal}.
Also, our ProxSVRG+ shows that the ``stochastically controlled'' trick of SCSG \citep{lei2017non}
(i.e., the length of each epoch is a geometrically distributed random variable)
is not really necessary for achieving the desired bound.\footnote{A similar observation was also made in Natasha1.5 \citep{allen2017natasha}.
	 However, Natasha1.5 divides each epoch into multiple sub-epochs and randomly chooses the iteration point at the end of each sub-epoch. In our ProxSVRG+, the length of an epoch is deterministic and it directly uses the last iteration point at the end of each epoch.
}
As a result, our straightforward ProxSVRG+ generalizes the result of SCSG to the more general nonsmooth nonconvex case and yields simpler analysis.

\section{Convergence Results}
\label{sec:convg}
Now, we present the main theorem for our ProxSVRG+ which corresponds to
the last two rows in Table \ref{tab:1} and give some remarks.

\begin{theorem}\label{thm:1}
Let step size $\eta=\frac{1}{6L}$ and $b$ denote the minibatch size. Then $\hx$ returned by Algorithm \ref{alg:1} is an $\epsilon$-accurate solution for problem (\ref{eq:form}) (i.e., $\E[\ns{\calG_\eta(\hx)}]\leq \epsilon$).
We distinguish the following two cases:
\begin{enumerate}[1)]
  \item We let batch size $B=n$. The number of SFO calls is at most
\begin{equation*}
36L\big(\Phi(x_0)-\Phi(x^*)\big)\Big(\frac{B}{\epsilon\sqrt{b}}+\frac{b}{\epsilon}\Big) =O\Big(\frac{n}{\epsilon\sqrt{b}}+\frac{b}{\epsilon}\Big).
\end{equation*}
  \item Under Assumption \ref{asp:var}, we
  let batch size $B=\min\{6\sigma^2/\epsilon,n\}$.
  The number of SFO calls is at most
\begin{equation*}
36L\big(\Phi(x_0)-\Phi(x^*)\big)\Big(\frac{B}{\epsilon\sqrt{b}}+\frac{b}{\epsilon}\Big) =O\Big(\big(n\wedge\frac{1}{\epsilon}\big)\frac{1}{\epsilon\sqrt{b}}+\frac{b}{\epsilon}\Big),
\end{equation*}
where $\wedge$ denotes the minimum.
\end{enumerate}
In both cases, the number of PO calls equals to the total number of iterations $T$, which is at most
$$\frac{36L}{\epsilon}\big(\Phi(x_0)-\Phi(x^*)\big)=O\left(\frac{1}{\epsilon}\right).$$
\end{theorem}

\topic{Remark}
  The proof for Theorem \ref{thm:1} is notably different from that of ProxSVRG \citep{reddi2016proximal}.
  \cite{reddi2016proximal} used a Lyapunov function $R_t^{s+1}=\Phi(x_t^{s+1})+c_t\ns{x_t^{s+1}-\tx^S}$ and showed that $R^s$ decreases by the accumulated gradient mapping $\sum_{t=1}^{m}\ns{\calG_\eta(x_t^s)}$ in epoch $s$. In our proof, we directly show that $\Phi(x^{s})$ decreases by $\sum_{t=1}^{m}\ns{\calG_\eta(x_t^s)}$ using a different analysis.
  This is made possible by tightening the inequalities using Young's inequality and Lemma \ref{lm:2} (which gives the relation between the variance of stochastic gradient estimator and the inner product of the gradient difference and point difference).
  Also, our convergence result holds for any minibatch size $b\in[1,n]$ unlike ProxSVRG $b\leq n^{2/3}$ (see Figure \ref{fig:1}). Moreover, ProxSVRG+ uses much less proximal oracle calls than ProxSVRG (see Figure \ref{fig:2}).

  For the online/stochastic Case 2),  we avoid the computation of the full gradient at the beginning of each epoch, i.e., $B\neq n$. Then, we use the similar idea in SCSG \citep{lei2017non} to bound the variance term, but we do not need the ``stochastically controlled'' trick of SCSG (as we discussed in Section \ref{sec:alg}) to achieve the desired convergence bound which yields a much simpler analysis for our ProxSVRG+.

We defer the proof of Theorem \ref{thm:1} to Appendix \ref{app:proof1}.
Also, similar convergence results for other choices of epoch length $m \neq \sqrt{b}$ are provided in Appendix \ref{app:epoch}.

\section{Convergence Under PL Condition}
\label{sec:pl}\vspace{-1mm}
In this section, we provide the global linear convergence rate for nonconvex functions under the Polyak-\L{}ojasiewicz (PL) condition \citep{polyak1963gradient}.
The original form of PL condition is
\begin{equation}\label{eq:pl}
\exists \mu>0, ~\mathrm{such~that} ~\ns{\nabla f(x)} \geq 2\mu (f(x)-f^*),~ \forall x,
\end{equation}
where $f^*$ denotes the (global) optimal function value.
It is worth noting that $f$ satisfies PL condition when $f$ is $\mu$-strongly convex.
Moreover, \cite{karimi2016linear} showed that PL condition is weaker than many conditions
(e.g., strong convexity (SC), restricted strong convexity (RSC) and weak strong convexity (WSC)
\citep{necoara2015linear}).
Also, if $f$ is convex, PL condition is equivalent to the error bounds (EB) and quadratic growth (QG) condition \citep{luo1993error,anitescu2000degenerate}.
Note that PL condition implies that every stationary point is a global minimum, but it does not imply there is a unique minimum unlike the strongly convex condition.

\topic{Further Motivation} In many cases, although the loss function is generally nonconvex,
the local region near a local minimum may satisfy the PL condition.
In fact, there are some recent studies showing
the strong convexity in the neighborhood of the ground truth solution in some simple neural networks
\citep{zhong2017recovery,fu2018local}.
Such results provide further motivation for studying the PL condition.
Moreover, we argue that our ProxSVRG+ is particularly desirable
in this case since it first converges sublinearly $O(1/\epsilon)$ (according to Theorem \ref{thm:1})
then \emph{automatically} converges linearly $O(\log1/\epsilon)$ (according to Theorem \ref{thm:pl1})
in the regions as long as the loss function satisfies the PL condition locally in these regions.
We list the convergence results in Table \ref{tab:3}  (also see the remarks after Theorem \ref{thm:pl1}).
\vspace{-4mm}
\begin{table}[!h]
 \centering
 \caption{Under the PL condition with parameter $\mu$}
 \label{tab:3}\vspace{1mm}
 \begin{tabular}{|c|c|c|c|}
  \hline
   	\multirow{2}*{Algorithms}  & Stochastic first-order & Proximal oracle   & Addi. \\
			& oracle (SFO)         &  (PO)                   & condition\\ \hline

  ProxGD~\citep{karimi2016linear}    & \multirow{2}*{$O(\frac{n}{\mu}\log\frac{1}{\epsilon})$} & \multirow{2}*{$O(\frac{1}{\mu}\log\frac{1}{\epsilon})$}
    & \multirow{2}*{--} \\
  (full gradient) &&& \\\hline

  ProxSVRG/SAGA
  & \multirow{2}*{$O\big(\frac{n}{\mu\sqrt{b}}\log\frac{1}{\epsilon}
  +n\log\frac{1}{\epsilon}\big)$}
  & \multirow{2}*{$O\big(\frac{n}{\mu b^{3/2}}\log\frac{1}{\epsilon}\big)$}
    & \multirow{2}*{$b\le n^{2/3}$} \\
   \citep{reddi2016proximal} &&&  \\\hline

  SCSG~\citep{lei2017non}
  &\multirow{3}*{$O\Big(\frac{b^{\frac{1}{3}}}{\mu}\big(n\wedge\frac{1}{\mu\epsilon}\big)^{\frac{2}{3}}
          \log\frac{1}{\epsilon}
        +\big(n\wedge\frac{1}{\mu\epsilon}\big)\log\frac{1}{\epsilon}\Big)$}
  &\multirow{3}*{NA}& \multirow{3}*{$\sigma = O(1)$} \\
   (smooth nonconvex, &&&\\
   i.e., $h(x)\equiv 0$) &&& \\\hline

  & \multirow{2}*{$O\big(\frac{n}{\mu\sqrt{b}}\log\frac{1}{\epsilon}
  +\frac{b}{\mu}\log\frac{1}{\epsilon}\big)$}
  & \multirow{2}*{$O(\frac{1}{\mu}\log\frac{1}{\epsilon})$}   & \multirow{2}*{--} \\
  \red{ProxSVRG+} &&& \\\cline{2-4}

   (this paper) &   \multirow{2}*{$O\Big(\big(n\wedge\frac{1}{\mu\epsilon}\big)\frac{1}{\mu\sqrt{b}}\log\frac{1}{\epsilon}
   +\frac{b}{\mu}\log\frac{1}{\epsilon}\Big)$}
  & \multirow{2}*{$O(\frac{1}{\mu}\log\frac{1}{\epsilon})$}  & \multirow{2}*{$\sigma = O(1)$} \\
  &&& \\\hline
 \end{tabular}

 \vspace{1mm}
 {\small
 The notation $\wedge$ denotes the minimum.
 Similar to Table \ref{tab:2},
 ProxSVRG+ is  better than ProxGD and ProxSVRG/SAGA, and generalizes the SCSG by choosing different minibatch size $b$.}\vspace{-4.5mm}
\end{table}

Due to the nonsmooth term $h(x)$ in problem (\ref{eq:form}),
we use the gradient mapping (see (\ref{eq:gradmap})) to define a more general form of PL condition as follows:
\begin{equation}\label{eq:gelpl}
\exists \mu>0, ~\mathrm{such~that} ~\ns{\calG_\eta(x)} \geq 2\mu (\Phi(x)-\Phi^*), ~\forall x.
\end{equation}
Recall that if $h(x)$ is a constant function, the gradient mapping reduces to  $\calG_\eta(x) = \nabla f(x)$. 
Our PL condition is different from the one used in \citep{karimi2016linear,reddi2016proximal}.
See the Remark (\ref{rem:2}) after Theorem \ref{thm:pl1}.

Similar to Theorem \ref{thm:1}, we provide the convergence result of ProxSVRG+ (Algorithm \ref{alg:1}) under PL-condition in the following Theorem \ref{thm:pl1}.
Note that under PL condition (i.e. (\ref{eq:gelpl}) holds), ProxSVRG+ can directly use the final iteration $\tx^S$ as the output point instead of the randomly chosen one $\hx$. Similar to \citep{reddi2016proximal}, we assume the condition number $L/\mu > \sqrt{n}$ for simplicity.
Otherwise, one can choose different step size $\eta$ which is similar to the case where
we deal with other choices of epoch length $m$ (see Appendix \ref{app:epoch}).

\begin{theorem}\label{thm:pl1}
Let step size $\eta=\frac{1}{6L}$ and $b$ denote the minibatch size. Then the final iteration point $\tx^S$ in Algorithm \ref{alg:1} satisfies $\E[\Phi(\tx^S)-\Phi^*]\leq \epsilon$ under PL condition.
We distinguish the following two cases:
\begin{enumerate}[1)]
  \item We let batch size $B=n$. The number of SFO calls is bounded by
\begin{equation*}
O\Big(\frac{n}{\mu\sqrt{b}}\log\frac{1}{\epsilon}
  +\frac{b}{\mu}\log\frac{1}{\epsilon}\Big).
\end{equation*}
  \item Under Assumption \ref{asp:var}, we let batch size $B=\min\{\frac{6\sigma^2}{\mu\epsilon},n\}$.
  The number of SFO calls is bounded by
\begin{equation*}
O\Big(\big(n\wedge\frac{1}{\mu\epsilon}\big)\frac{1}{\mu\sqrt{b}}\log\frac{1}{\epsilon}
   +\frac{b}{\mu}\log\frac{1}{\epsilon}\Big),
\end{equation*}where $\wedge$ denotes the minimum.
\end{enumerate}
In both cases, the number of PO calls equals to the total number of iterations $T$ which is bounded by $$O\Big(\frac{1}{\mu}\log\frac{1}{\epsilon}\Big).$$
\end{theorem}

\topic{Remark}
\begin{enumerate}[(1)]
  \item We show that ProxSVRG+ directly obtains a global linear convergence rate without restart by a nontrivial proof. Note that \cite{reddi2016proximal}
      used PL-SVRG/SAGA to restart ProxSVRG/SAGA $O(\log(1/\epsilon))$ times to obtain the linear convergence rate under PL condition.

      Moreover, similar to Table \ref{tab:2}, if we choose $b=1$ or $n$ for ProxSVRG+, then its convergence result is $O(\frac{n}{\mu}\log\frac{1}{\epsilon})$, which is the same as ProxGD \citep{karimi2016linear}.
      If we choose $b=n^{2/3}$ for ProxSVRG+, then the convergence result is $O(\frac{n^{2/3}}{\mu}\log\frac{1}{\epsilon})$, the same as the best result achieved by ProxSVRG/SAGA \citep{reddi2016proximal}.
      If we choose $b=1/(\mu\epsilon)^{2/3}$ (assuming $1/(\mu\epsilon)<n$) for ProxSVRG+, then its convergence result is $O(\frac{1}{\mu^{5/3}\epsilon^{2/3}}\log\frac{1}{\epsilon})$
      which generalizes the best result of SCSG \citep{lei2017non} to the
      more general nonsmooth nonconvex case and is better than ProxGD and ProxSVRG/SAGA.
      Also note that our ProxSVRG+ uses much less proximal oracle calls than ProxSVRG/SAGA if $b<n^{2/3}$.
  \item Another benefit of ProxSVRG+ is that it can \emph{automatically} switch to the faster linear convergence rate in some regions as long as the loss function satisfies the PL condition locally in these regions. This is impossible for ProxSVRG~\citep{reddi2016proximal} since it needs to be
  restarted many times.
  \item \label{rem:2}
  We want to point out that \citep{karimi2016linear,reddi2016proximal} used the following form of
  PL condition:
\begin{equation}\label{eq:plreddi}
\exists \mu>0, ~\mathrm{such~that} ~D_h(x,\alpha) \geq 2\mu (\Phi(x)-\Phi^*), ~\forall x,
\end{equation}
where $D_h(x,\alpha) := -2\alpha\min_y\big\{\inner{\nabla f(x)}{y-x}+\frac{\alpha}{2}\ns{y-x}+h(y)-h(x)\big\}$.
    Our PL condition is arguably more natural.
    In fact, one can show that if $\alpha=1/\eta$,
    our new PL condition \eqref{eq:gelpl} implies \eqref{eq:plreddi}.
    For a direct comparison with prior results,
    we also provide the proof of the same result of Theorem \ref{thm:pl1}
    under PL condition (\ref{eq:plreddi}) in the appendix.
\end{enumerate}
The proofs of Theorem \ref{thm:pl1} under PL form (\ref{eq:gelpl}) and (\ref{eq:plreddi}) are provided in Appendix \ref{app:pl} and \ref{app:pl2}, respectively.
Recently, \cite{csiba2017global} proposed a novel weakly PL condition.
The (strongly) PL condition (\ref{eq:gelpl}) or (\ref{eq:plreddi}) serves as a generalization of strong convexity as we discussed in the beginning of this section. One can achieve linear convergence under (\ref{eq:gelpl}) or (\ref{eq:plreddi}). However, the weakly PL condition \citep{csiba2017global} may be considered as a generalization of (weak) convexity. Although one only achieves the sublinear convergence under this condition, it is still interesting to figure out similar (sublinear) convergence (for ProxSVRG+, ProxSVRG, etc.) under their weakly PL condition.

\section{Experiments}
\label{sec:exp}\vspace{-1mm}
In this section, we present the experimental results.
We compare the nonconvex ProxSVRG+ with nonconvex ProxGD, ProxSGD \citep{ghadimi2016mini}, ProxSVRG \citep{reddi2016proximal}.
We conduct the experiments using the non-negative principal component analysis (NN-PCA) problem (same as \citep{reddi2016proximal}).
In general, NN-PCA is NP-hard.
Specifically, the optimization problem for a given set of samples (i.e., $\{z_i\}_{i=1}^n$) is:
\begin{equation}\label{eq:nnpca}
\min_{\n{x}\leq 1, x\geq0} -\frac{1}{2}x^T\Big(\sum_{i=1}^n{z_iz_i^T}\Big)x.
\end{equation}
Note that (\ref{eq:nnpca}) can be written in the form (\ref{eq:form}),
where $f(x)=\sum_{i=1}^n{f_i(x)}=\sum_{i=1}^n{-\frac{1}{2}(x^Tz_i)^2}$
and $h(x)=I_C(x)$ where set $C=\{x\in \R^d | \n{x}\leq 1, x\geq 0\}$.
We conduct the experiment on the standard MNIST and `a9a' datasets.~\footnote{The
 datasets can be downloaded from \url{https://www.csie.ntu.edu.tw/~cjlin/libsvmtools/datasets/}}
The experimental results on both datasets (corresponding to the first row and second row in Figure \ref{fig:n3}--\ref{fig:n5}) are almost the same.

The samples from each dataset are normalized, i.e., $\n{z_i}=1$ for all $i\in n$.
The parameters of the algorithms are chosen as follows:
$L$ can be precomputed from the data samples $\{z_i\}_{i=1}^n$ in the same way as in
\citep{Li2017convergence}.
The step sizes $\eta$ for different algorithms are set to be the ones used in their convergence results:
For ProxGD, it is $\eta=1/L$ (see Corollary 1 in \citep{ghadimi2016mini});
for ProxSGD, $\eta=1/(2L)$ (see Corollary 3 in \citep{ghadimi2016mini});
for ProxSVRG, $\eta=b^{3/2}/(3Ln)$ (see Theorem 6 in \citep{reddi2016proximal}).
The step size for our ProxSVRG+ is $1/(6L)$ (see our Theorem \ref{thm:1}).
We did not further tune the step sizes. The batch size $B$ (in Line 4 of Algorithm \ref{alg:1}) is equal to $n/5$ (i.e., 20\% data samples). We also considered $B=n/10$, the performance among these algorithms are similar to the case $B=n/5$. In practice, one can tune the step size $\eta$ and parameter $B$.

\eat{
Regarding the comparison among these algorithms, we use the number of SFO calls (see Definition \ref{dfn:oracle}) to evaluate them since the number of PO calls of them are the same except ProxSVRG (which is already clearly demonstrated by Figure \ref{fig:2}).
Note that we amortize the batch size ($n$ or $B$ in Line 4 of Algorithm \ref{alg:1}) into the inner loops, so that the curves in the figures are smoother,
i.e., the number of SFO calls for each iteration (inner loop) of ProxSVRG and ProxSVRG+ is counted as $b+n/m$ and $b+B/m$, respectively. Note that ProxGD uses $n$ SFO calls (full gradient) in each iteration.}
\vspace{-3mm}
\begin{figure}[!h]
\centering
  \begin{minipage}{\textwidth}
        \includegraphics[width=0.33\textwidth]{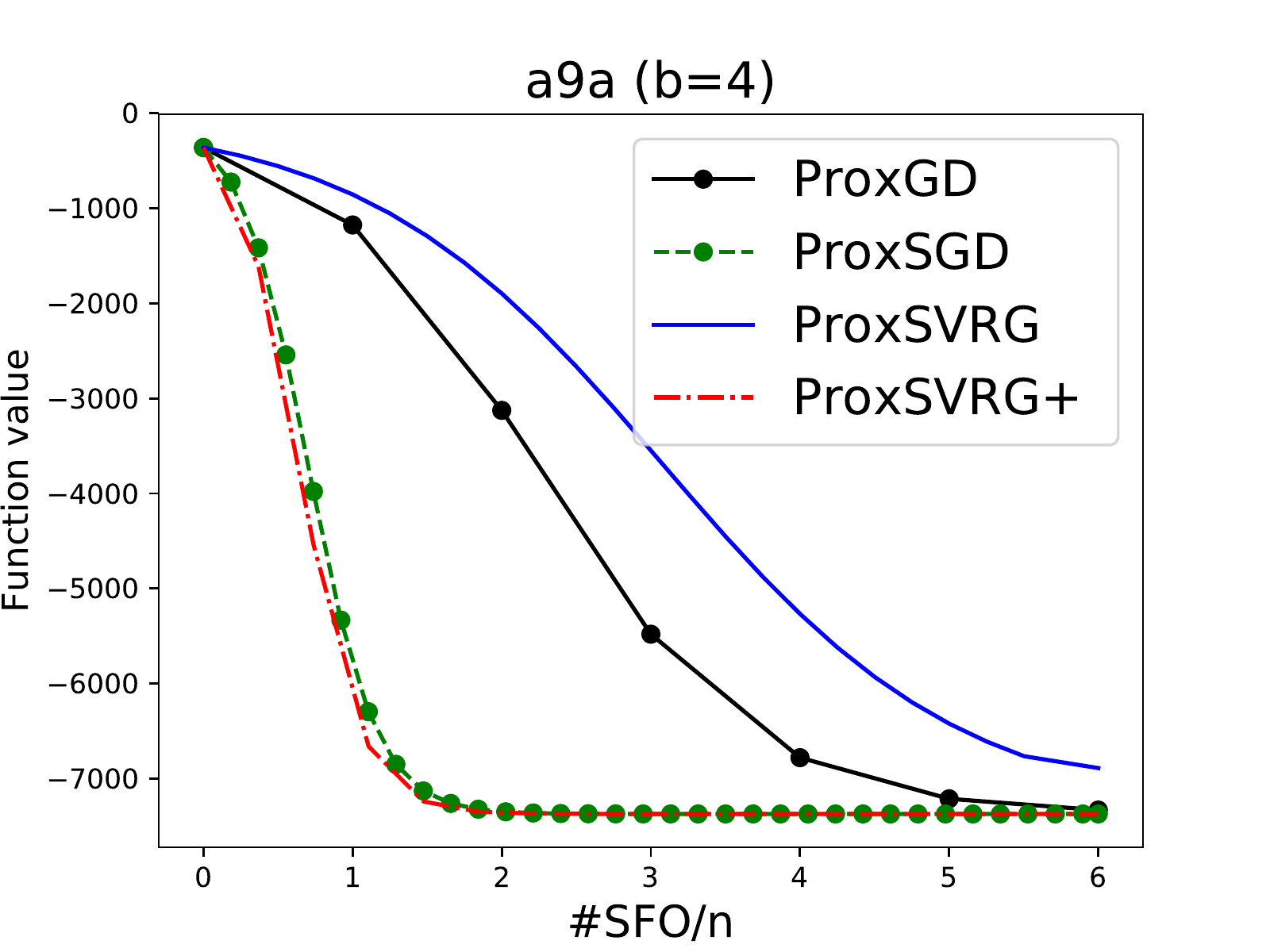}
        \includegraphics[width=0.33\textwidth]{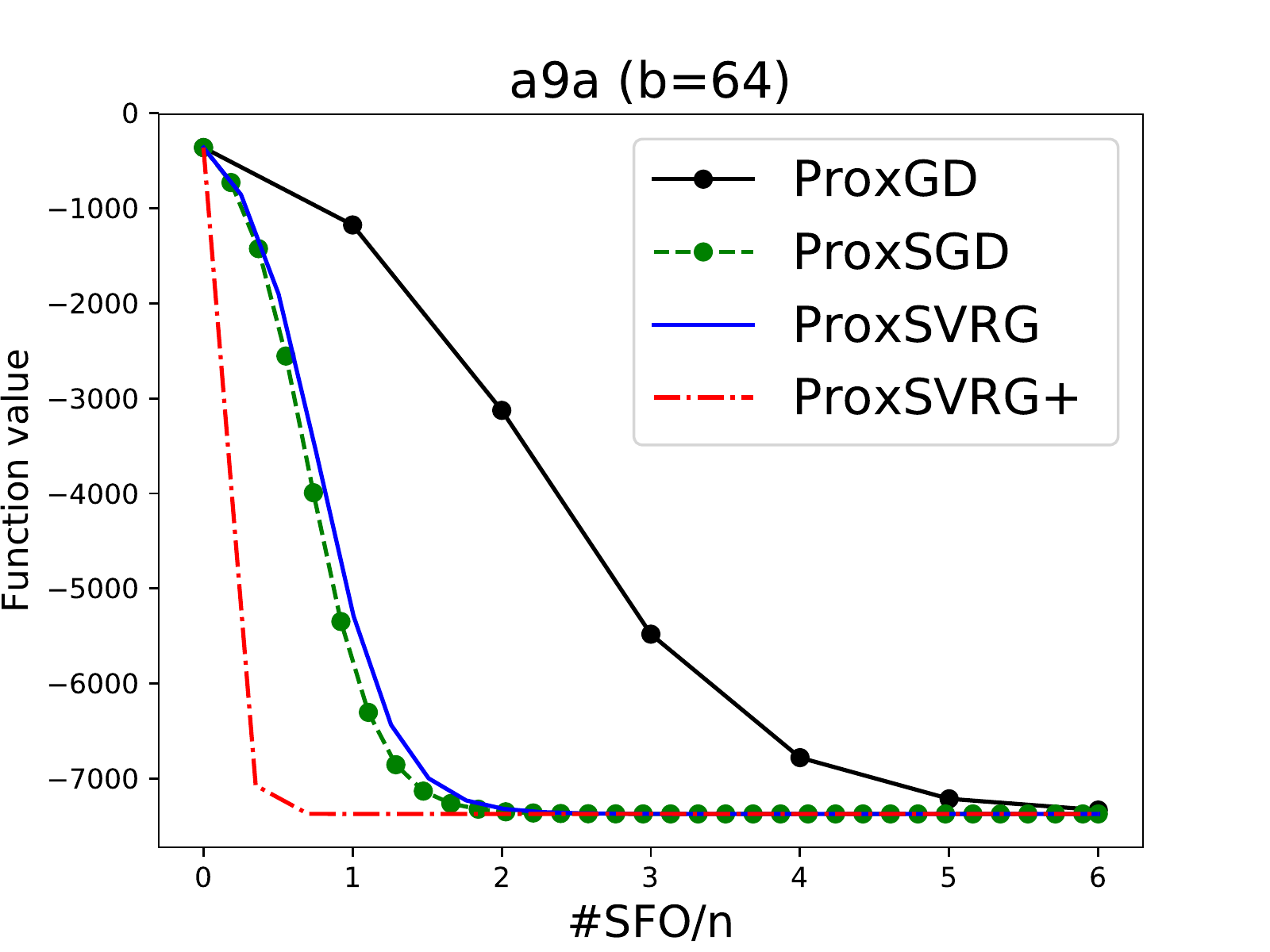}
        \includegraphics[width=0.33\textwidth]{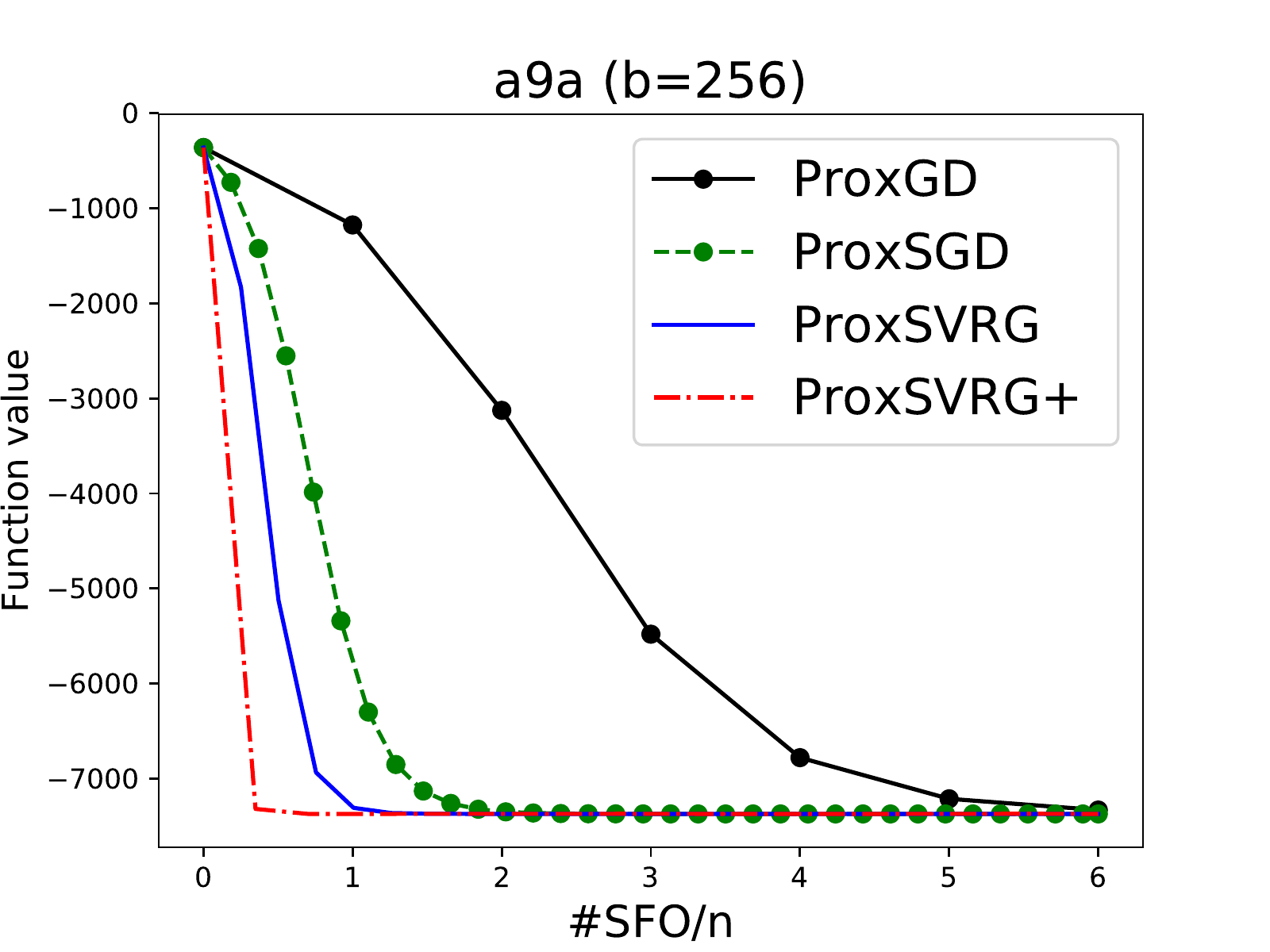}
  \end{minipage}\vspace{-3mm}
  \begin{minipage}{\textwidth}
        \centering
        \includegraphics[width=0.33\textwidth]{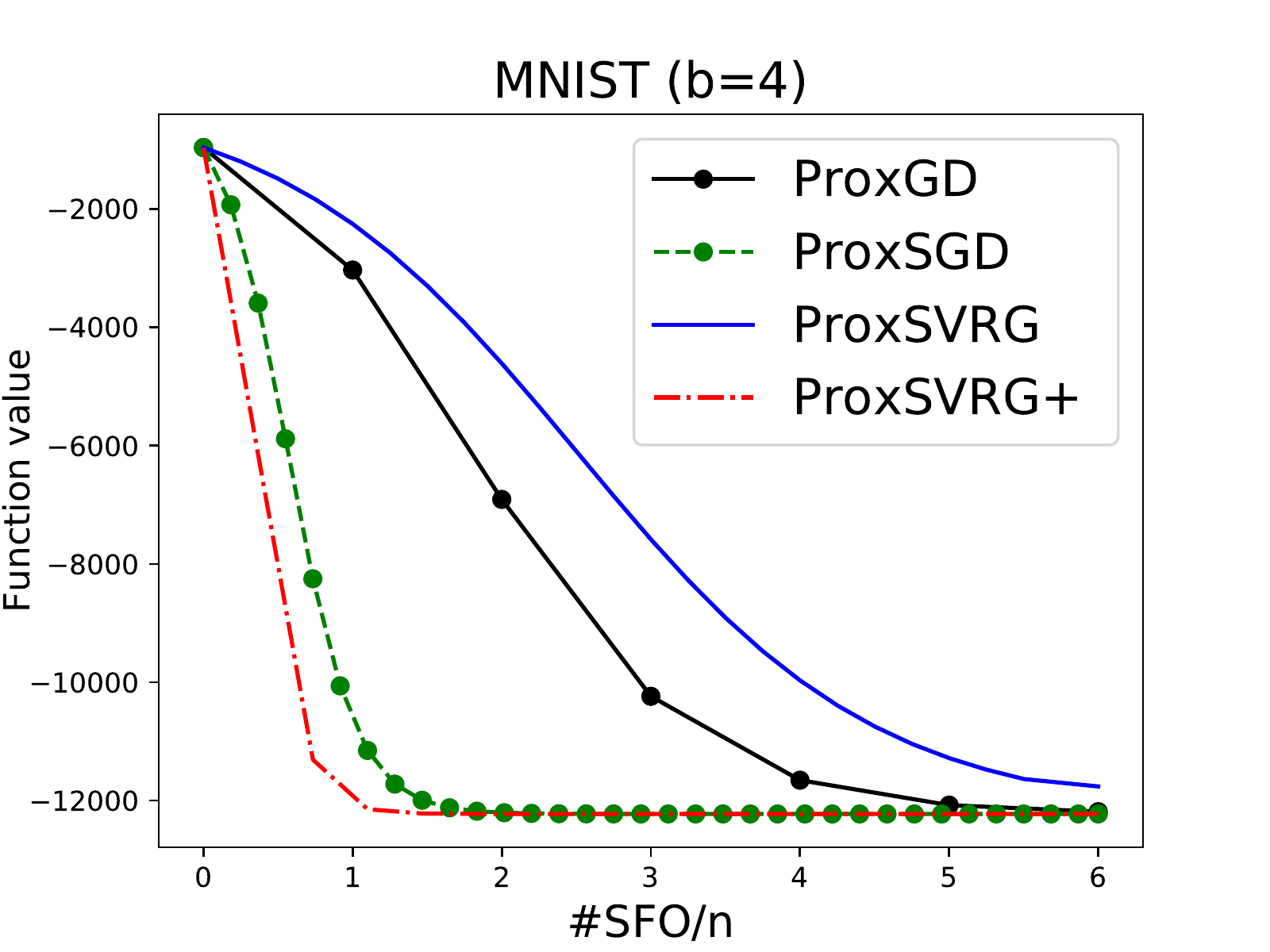}
        \includegraphics[width=0.33\textwidth]{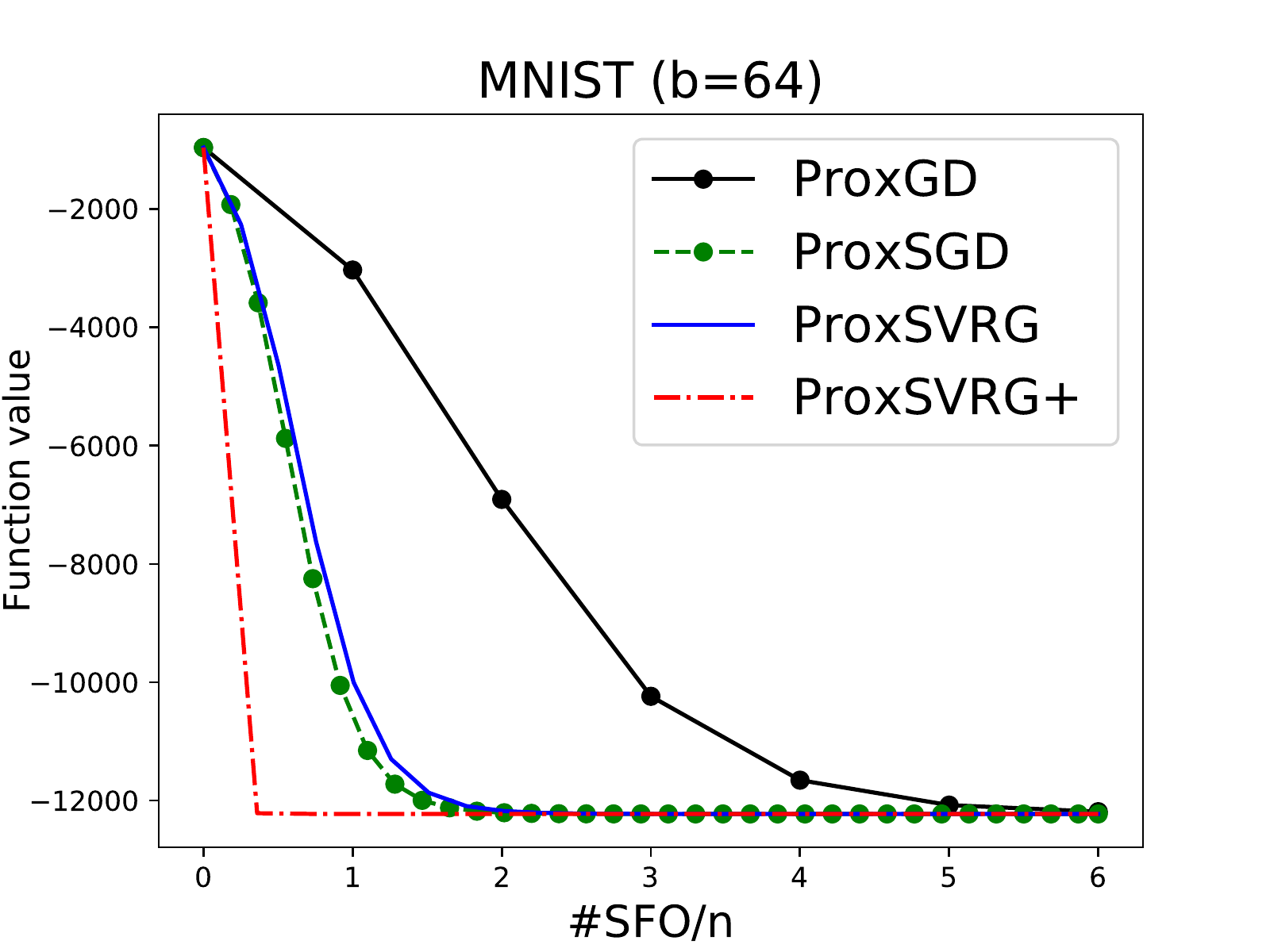}
        \includegraphics[width=0.33\textwidth]{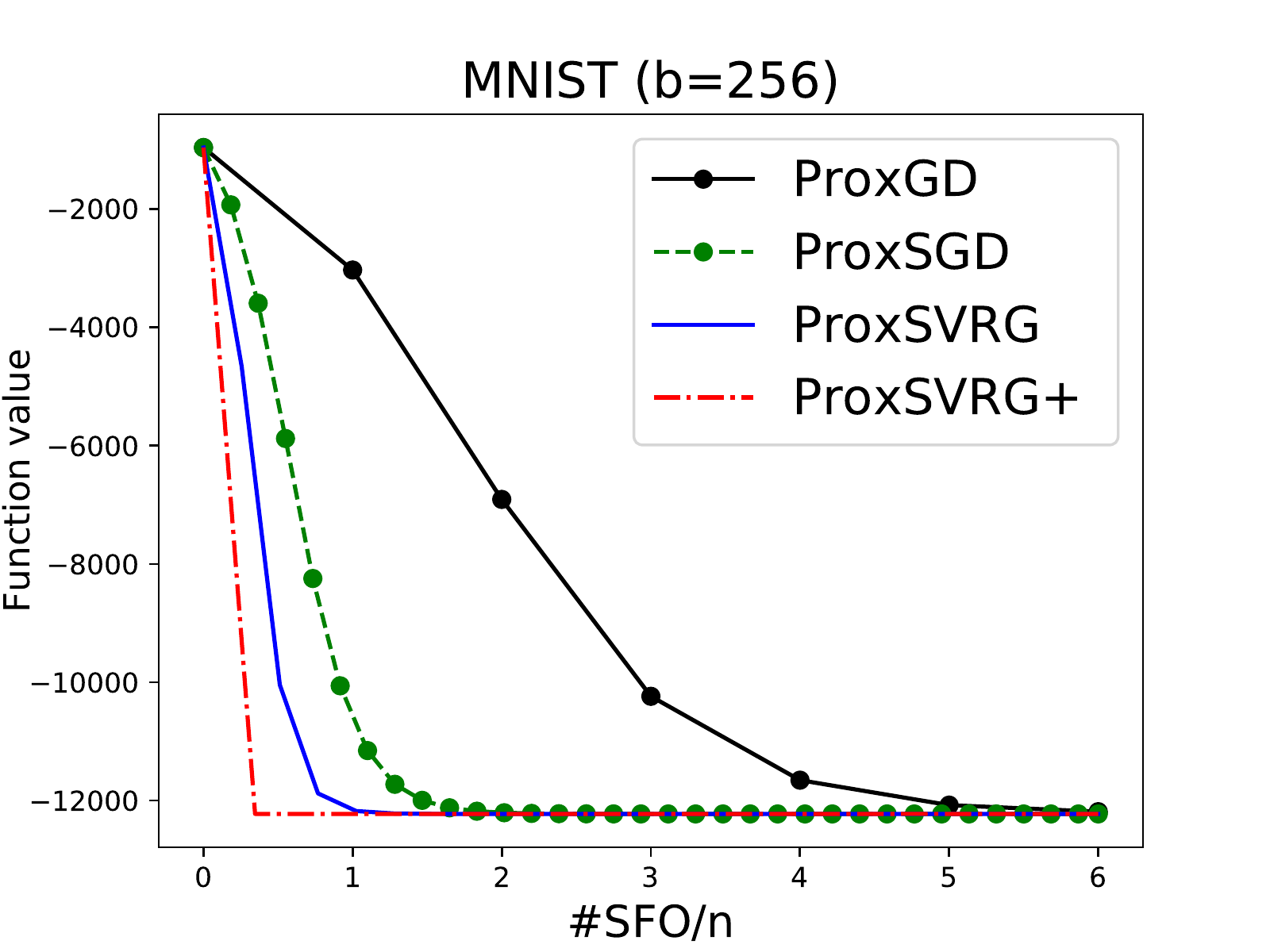}
  \end{minipage}\vspace{-3mm}
    \caption{Comparison among algorithms with different minibatch size $b$}
     \label{fig:n3}
\end{figure}\vspace{-5mm}

\newpage
In Figure \ref{fig:n3},
we compare the performance of these four algorithms
as we vary the minibatch size $b$.
In particular, the first column ($b=4$) shows that ProxSVRG+ and ProxSVRG perform similar to ProxSGD and ProxGD respectively, which is quite consistent with the theoretical results (Figure \ref{fig:1}). Then, ProxSVRG+ and ProxSVRG both get better as $b$ increases. Note that our ProxSVRG+ performs better than ProxGD, ProxSGD and ProxSVRG.
\vspace{-4mm}
\begin{figure}[!h]
\centering
\begin{minipage}{0.66\textwidth}
        \centering
        \includegraphics[width=0.496\textwidth]{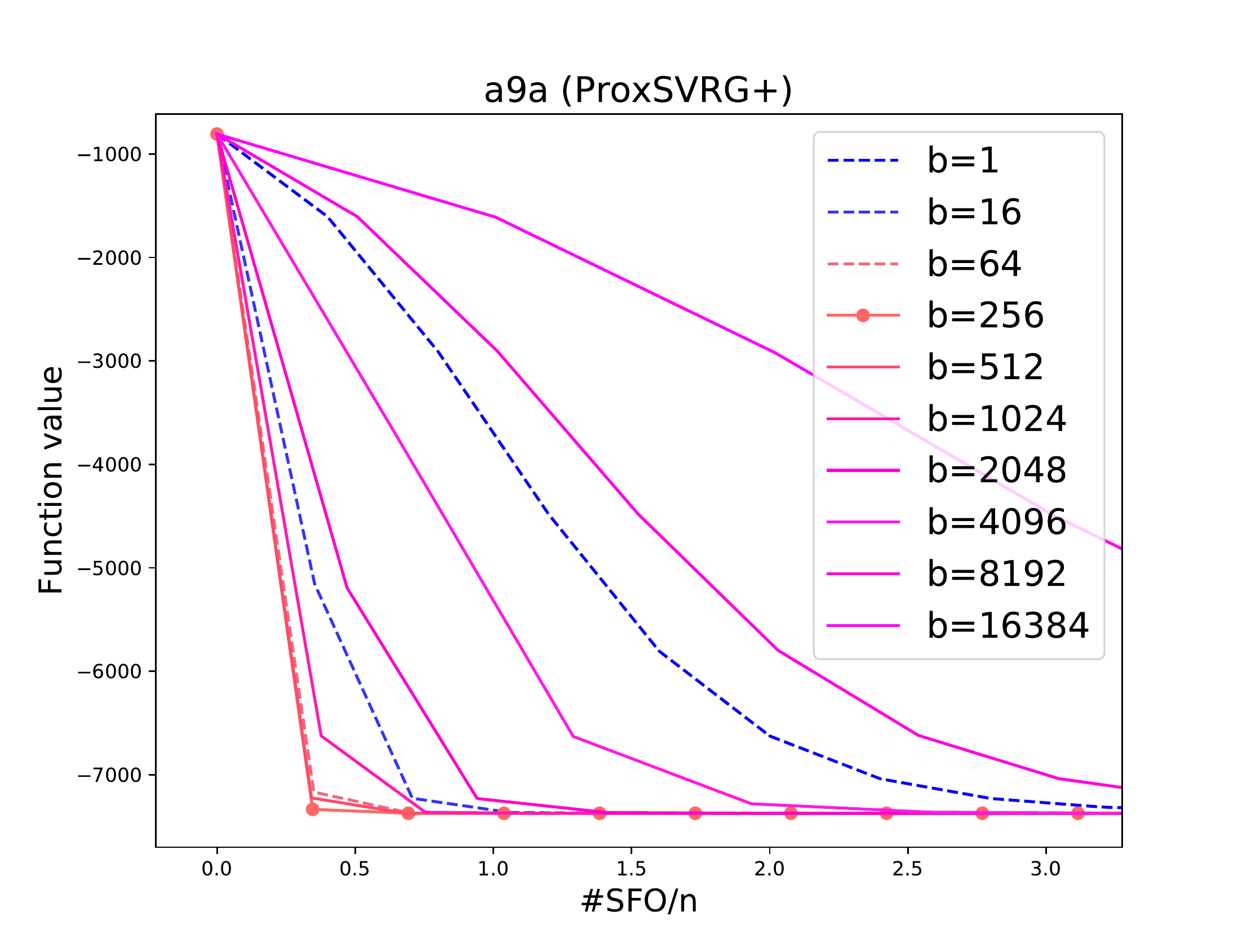}
        \includegraphics[width=0.496\textwidth]{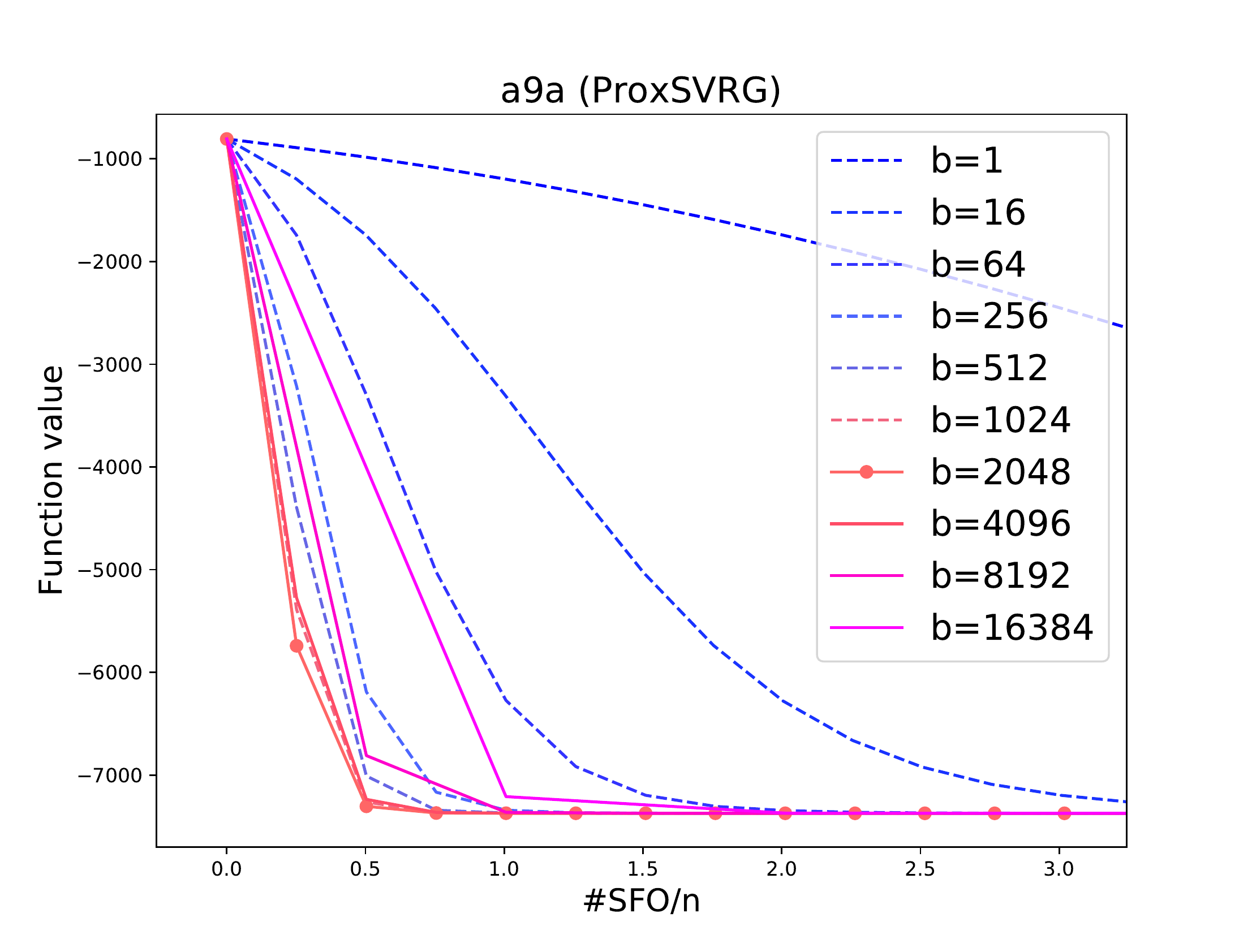}
  \end{minipage}\vspace{-1.7mm}
  \begin{minipage}{0.33\textwidth}
        \centering
        \includegraphics[width=\textwidth]{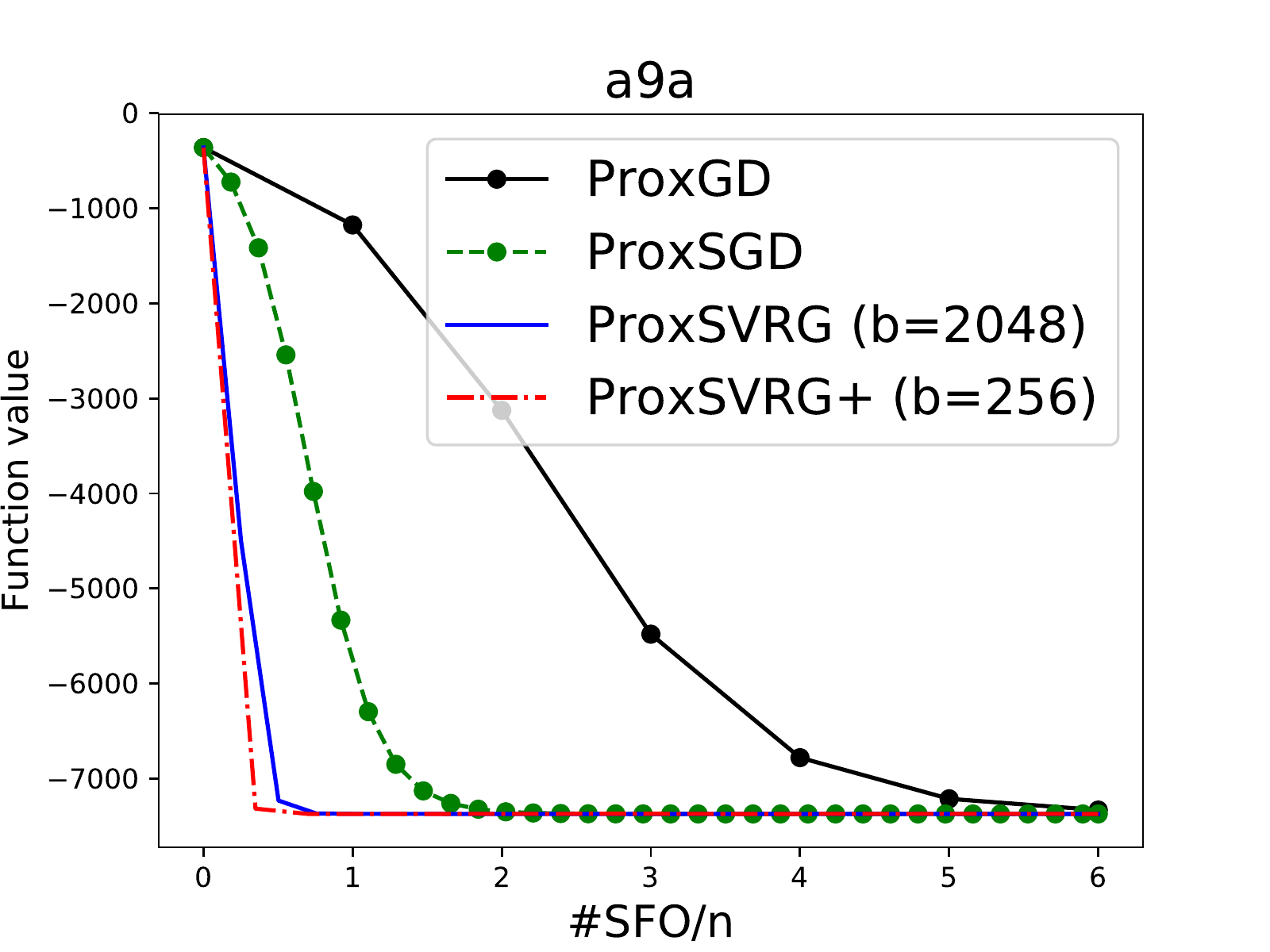}
  \end{minipage}\vspace{-1.7mm}
  \begin{minipage}{0.66\textwidth}
        \centering
        \includegraphics[width=0.496\textwidth]{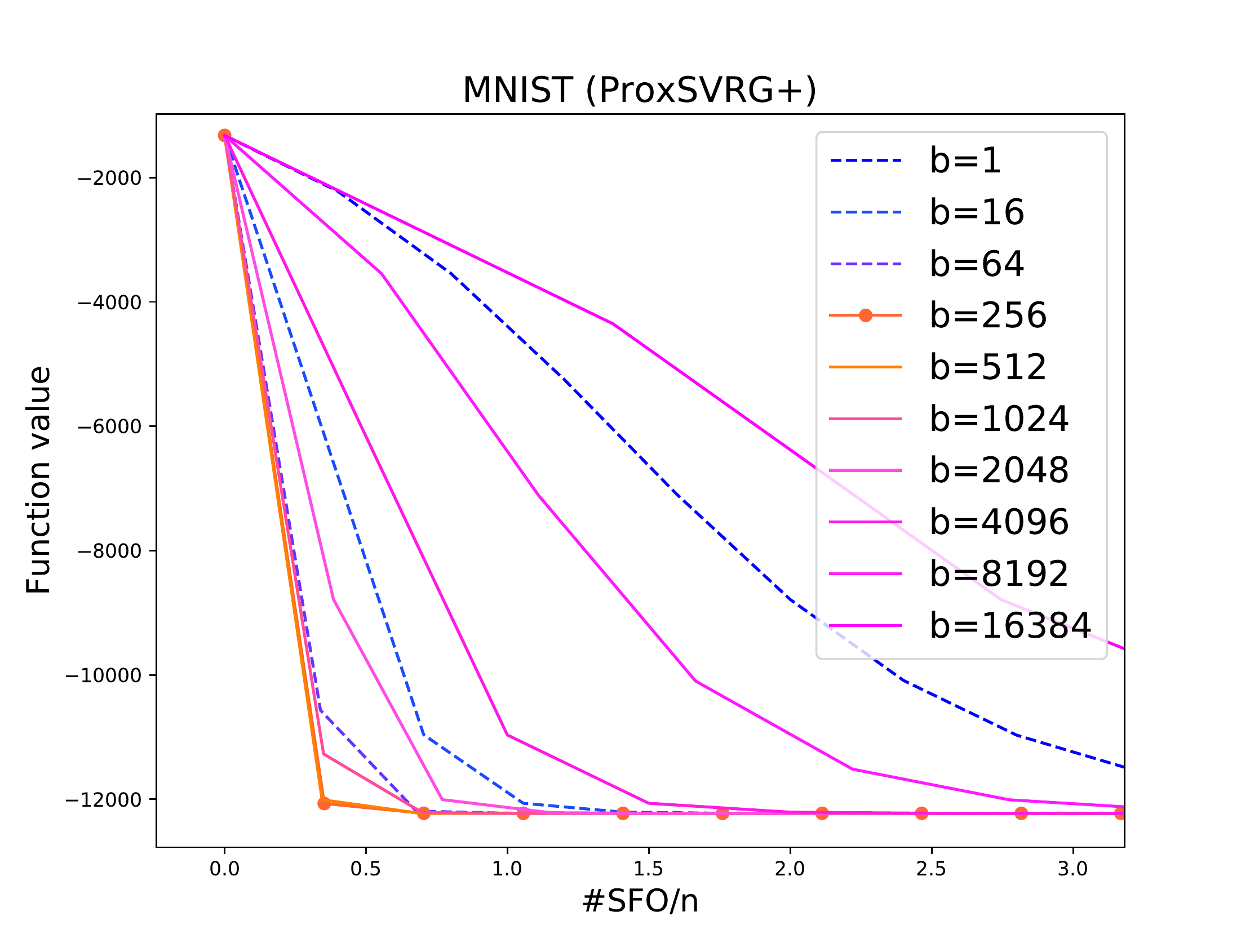}
        \includegraphics[width=0.496\textwidth]{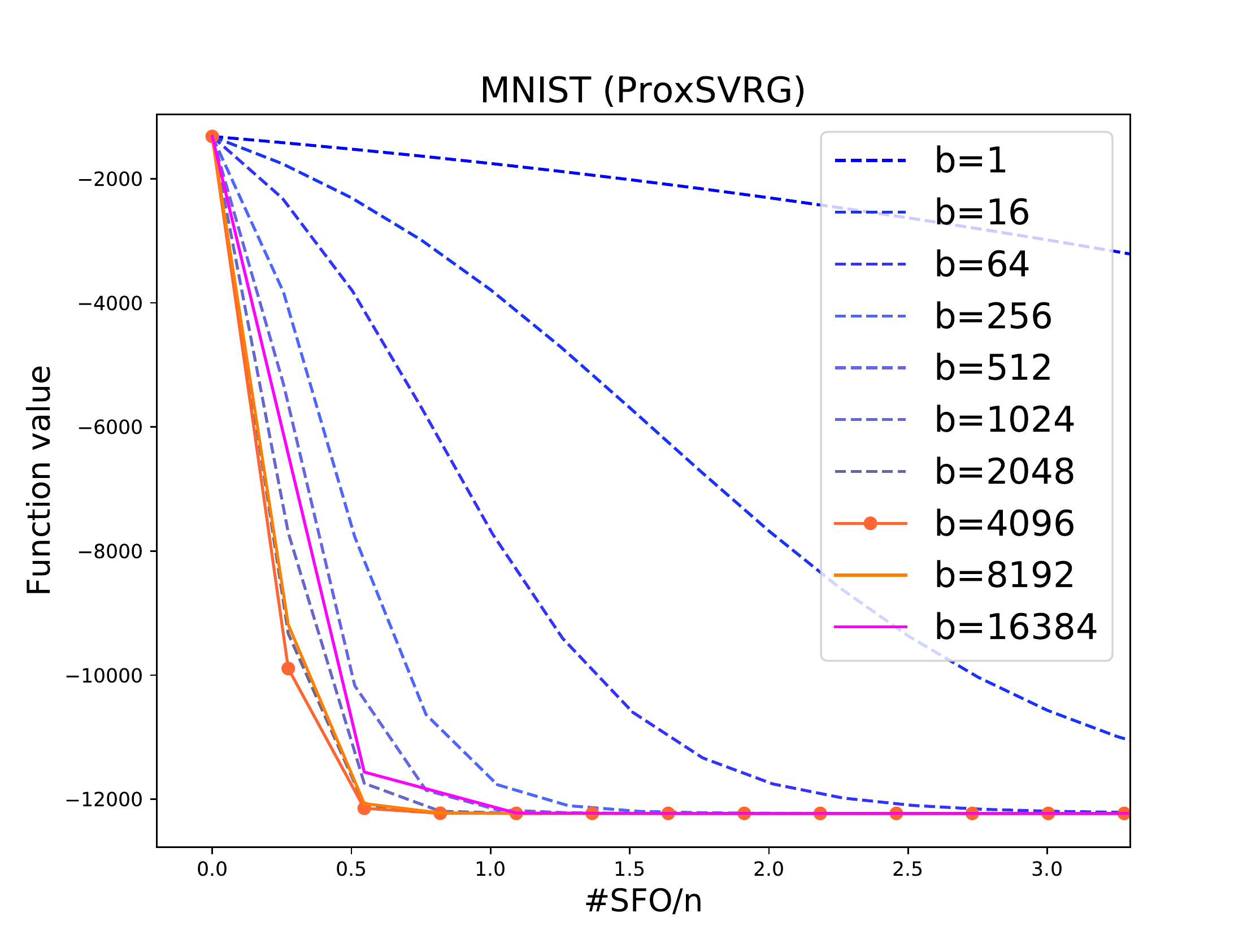}
        \vspace{-7mm}
  \caption{ProxSVRG+ and ProxSVRG under different $b$}
  \label{fig:n4}
  \end{minipage}\vspace{-1.7mm}
  \begin{minipage}{0.33\textwidth}
        \centering
        \includegraphics[width=\textwidth]{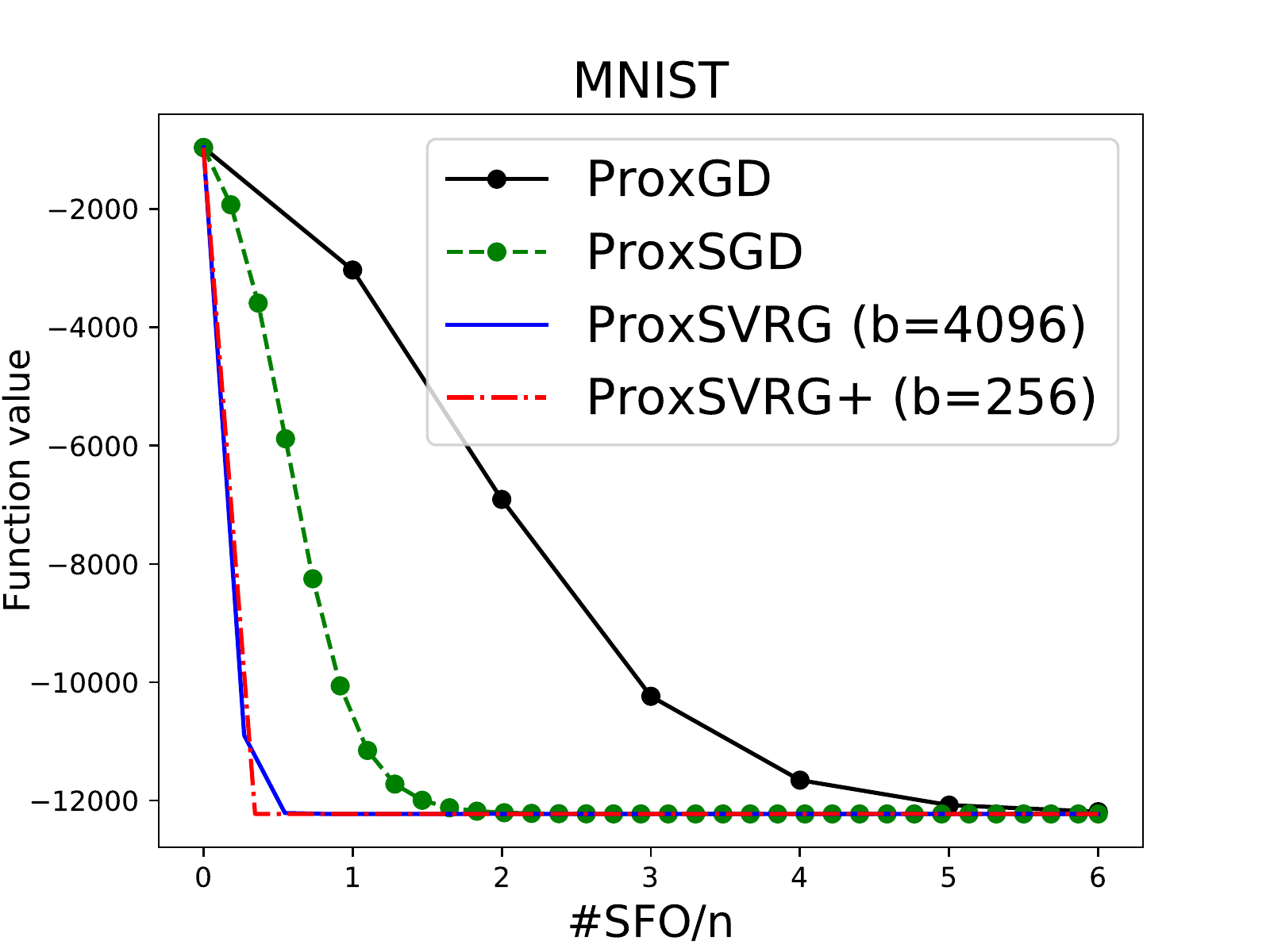}
        \vspace{-7mm}
        \caption{Under the best $b$}
        \label{fig:n5}
  \end{minipage}\vspace{-1mm}
\end{figure}

Figure \ref{fig:n4} demonstrates that our ProxSVRG+ prefers smaller minibatch sizes than ProxSVRG (see the curves with dots). Then, in Figure \ref{fig:n5}, we compare the algorithms with their corresponding best minibatch size $b$.

In conclusion, the experimental results are quite consistent with the theoretical results, i.e., different algorithms favor different minibatch sizes (see Figure \ref{fig:1}).
Concretely, our ProxSVRG+ achieves its best performance with a moderate minibatch size $b=256$ unlike ProxSVRG with $b=2048/4096$.
Besides, choosing $b=64$ is already good enough for ProxSVRG+ by  comparing the second column and last column of Figure \ref{fig:n3}, however ProxSVRG is only as good as ProxSGD with such a minibatch size.
Moreover, ProxSVRG+ uses much less proximal oracle calls than ProxSVRG if $b<n^{2/3}$ (see Figure \ref{fig:2}).
Note that small minibatch size also usually provides better generalization in practice.
Thus, we argue that our ProxSVRG+ might be more attractive in certain applications due to its moderate minibatch size.

\vspace{-3mm}
\section{Conclusion}\vspace{-2mm}
\label{sec:con}
In this paper, we propose a simple proximal stochastic method called ProxSVRG+ for nonsmooth nonconvex optimization.
We prove that ProxSVRG+ improves/generalizes several well-known convergence results (e.g., ProxGD, ProxSGD, ProxSVRG/SAGA and SCSG) by choosing proper minibatch sizes.
In particular, ProxSVRG+ is $\sqrt{b}$ (or $\sqrt{b}\epsilon n$ if $n>1/\epsilon$) times faster than ProxGD, which partially answers the open problem (i.e., developing stochastic methods with provably better performance than ProxGD with constant minibatch size $b$) proposed in \citep{reddi2016proximal}.
Also, ProxSVRG+ generalizes the results of SCSG \citep{lei2017non} to this nonsmooth nonconvex case, and it is more straightforward than SCSG and yields simpler analysis.
Moreover, for nonconvex functions satisfying Polyak-\L{}ojasiewicz condition,
we prove that ProxSVRG+ achieves the global linear convergence rate without restart.
As a result, ProxSVRG+ can \emph{automatically} switch to the faster linear convergence rate (i.e., $O(\log 1/\epsilon)$) from sublinear convergence rate (i.e., $O(1/\epsilon)$) in some regions (e.g., the neighborhood of a local minimum) as long as the objective function satisfies the PL condition locally in these regions.
This is impossible for ProxSVRG~\citep{reddi2016proximal} since it needs to be restarted $O(\log 1/\epsilon)$ times.

\vspace{-2mm}
\subsection*{Acknowledgments}\vspace{-1mm}
The authors would like to thank Rong Ge, Xiangliang Zhang and the anonymous reviewers for useful suggestions.

\newpage
\bibliographystyle{plainnat}
\bibliography{draft}

\newpage
\appendix

\section{Proofs for Nonconvex ProxSVRG+ Algorithm}
\label{app:proof}

In this appendix, we first provide the proof of Theorem \ref{thm:1} (Appendix \ref{app:proof1}).
Then we provide the proof for other choices of epoch length $m$ (Appendix \ref{app:epoch}).

\vspace{1mm}
\subsection{Proof of Theorem \ref{thm:1}}
\label{app:proof1}

Before proving Theorem \ref{thm:1}, we need a useful lemma for the proximal operator.

\begin{lemma}\label{lm:key}
Let $x^+ := \prox_{\eta h}(x-\eta v)$, then the following inequality holds:
\begin{equation}\label{eq:key}
\Phi(x^+)\leq \Phi(z) + \inner{\nabla f(x)-v}{x^+-z} -\frac{1}{\eta}\inner{x^+-x}{x^+-z}
              +\frac{L}{2}\ns{x^+-x}+\frac{L}{2}\ns{z-x}, ~~\forall z\in\R^d.
\end{equation}
\end{lemma}
\begin{proof}
First, we recall the proximal operator (see (\ref{eq:prox2})):
\begin{equation}\label{eq:proxv}
\prox_{\eta h}(x-\eta v) := \arg\min_{y\in\R^d}{\Big(h(y)+\frac{1}{2\eta}\ns{y-x}+\inner{v}{y}\Big)}.
\end{equation}
For the nonsmooth function $h(x)$, we have
\begin{align}
h(x^+) &\leq h(z) + \inner{p}{x^+-z} \label{eq:h1}\\
       &= h(z) - \inner{v+\frac{1}{\eta}(x^+-x)}{x^+-z}, \label{eq:h2}
\end{align}
where $p\in\partial h(x^+)$ such that $p+\frac{1}{\eta}(x^+-x)+v=0$ according to the optimality condition of $(\ref{eq:proxv})$, and (\ref{eq:h1}) holds due to the convexity of $h$.

For the nonconvex function $f(x)$, we have
\begin{align}
f(x^+) &\leq f(x) + \inner{\nabla f(x)}{x^+-x} +\frac{L}{2}\ns{x^+-x} \label{eq:f1}\\
-f(z)  &\leq -f(x) + \inner{-\nabla f(x)}{z-x} +\frac{L}{2}\ns{z-x}, \label{eq:f2}
\end{align}
where (\ref{eq:f1}) holds since $f(x)$ has $L$-Lipschitz continuous gradient (see (\ref{eq:smooth})), and (\ref{eq:f2}) holds since $-f(x)$ has the same $L$-Lipschitz continuous gradient as $f(x)$.

This lemma is proved by adding (\ref{eq:h2}), (\ref{eq:f1}), (\ref{eq:f2}), and recalling $\Phi(x) = f(x) +h(x)$.
\end{proof}

\vspace{3mm}
\begin{proofof}{Theorem \ref{thm:1}}
Now, we are ready to use Lemma \ref{lm:key} to prove Theorem \ref{thm:1}.
Let $x_t^s := \prox_{\eta h}(x_{t-1}^s-\eta v_{t-1}^s)$ and $\bx_t^s := \prox_{\eta h}\big(x_{t-1}^s-\eta \nabla f(x_{t-1}^s)\big)$.
By letting $x^+=x_t^s, x=x_{t-1}^s, v=v_{t-1}^s$ and $z=\bx_t^s$ in (\ref{eq:key}), we have
\begin{equation}\label{eq:3}
\Phi(x_t^s)\leq \Phi(\bx_t^s) + \inner{\nabla f(x_{t-1}^s)-v_{t-1}^s}{x_t^s-\bx_t^s} -\frac{1}{\eta}\inner{x_t^s-x_{t-1}^s}{x_t^s-\bx_t^s} +\frac{L}{2}\ns{x_t^s-x_{t-1}^s}+\frac{L}{2}\ns{\bx_t^s-x_{t-1}^s}.
\end{equation}
Besides, by letting $x^+=\bx_t^s, x=x_{t-1}^s, v=\nabla f(x_{t-1}^s)$ and $z=x=x_{t-1}^s$ in (\ref{eq:key}), we have
\begin{align}
\Phi(\bx_t^s)\leq \Phi(x_{t-1}^s) -\frac{1}{\eta}\inner{\bx_t^s-x_{t-1}^s}{\bx_t^s-x_{t-1}^s}
              +\frac{L}{2}\ns{\bx_t^s-x_{t-1}^s}
              =  \Phi(x_{t-1}^s) - \Big(\frac{1}{\eta}-\frac{L}{2}\Big)\ns{\bx_t^s-x_{t-1}^s}. \label{eq:4}
\end{align}
We add (\ref{eq:3}) and (\ref{eq:4}) to obtain the key inequality
\begin{align}
  \Phi(x_t^s) &\leq \Phi(x_{t-1}^s) + \frac{L}{2}\ns{x_t^s-x_{t-1}^s} -\Big(\frac{1}{\eta}-L\Big)\ns{\bx_t^s-x_{t-1}^s}
  + \inner{\nabla f(x_{t-1}^s)-v_{t-1}^s}{x_t^s-\bx_t^s} \notag \\
  &\qquad \quad -\frac{1}{\eta}\inner{x_t^s-x_{t-1}^s}{x_t^s-\bx_t^s}   \notag\\
   &= \Phi(x_{t-1}^s) + \frac{L}{2}\ns{x_t^s-x_{t-1}^s} -\Big(\frac{1}{\eta}-L\Big)\ns{\bx_t^s-x_{t-1}^s}
   + \inner{\nabla f(x_{t-1}^s)-v_{t-1}^s}{x_t^s-\bx_t^s} \notag \\
  &\qquad \quad -\frac{1}{2\eta}\big(\ns{x_t^s-x_{t-1}^s}+\ns{x_t^s-\bx_t^s}-\ns{\bx_t^s-x_{t-1}^s}\big)\notag\\
  &= \Phi(x_{t-1}^s) -\Big(\frac{1}{2\eta}-\frac{L}{2}\Big)\ns{x_t^s-x_{t-1}^s} -\Big(\frac{1}{2\eta}-L\Big)\ns{\bx_t^s-x_{t-1}^s}
   + \inner{\nabla f(x_{t-1}^s)-v_{t-1}^s}{x_t^s-\bx_t^s} \notag \\
  &\qquad \quad -\frac{1}{2\eta}\ns{x_t^s-\bx_t^s} \notag\\
  &\leq \Phi(x_{t-1}^s) -\Big(\frac{1}{2\eta}-\frac{L}{2}\Big)\ns{x_t^s-x_{t-1}^s} -\Big(\frac{1}{2\eta}-L\Big)\ns{\bx_t^s-x_{t-1}^s}
   + \inner{\nabla f(x_{t-1}^s)-v_{t-1}^s}{x_t^s-\bx_t^s} \notag \\
  &\qquad \quad -\frac{1}{8\eta}\ns{x_t^s-x_{t-1}^s}+\frac{1}{6\eta}\ns{\bx_t^s-x_{t-1}^s} \label{eq:beta}\\
  &= \Phi(x_{t-1}^s) -\Big(\frac{5}{8\eta}-\frac{L}{2}\Big)\ns{x_t^s-x_{t-1}^s} -\Big(\frac{1}{3\eta}-L\Big)\ns{\bx_t^s-x_{t-1}^s}
   + \inner{\nabla f(x_{t-1}^s)-v_{t-1}^s}{x_t^s-\bx_t^s} \notag \\
   &\leq \Phi(x_{t-1}^s) -\Big(\frac{5}{8\eta}-\frac{L}{2}\Big)\ns{x_t^s-x_{t-1}^s} -\Big(\frac{1}{3\eta}-L\Big)\ns{\bx_t^s-x_{t-1}^s}
   + \eta\ns{\nabla f(x_{t-1}^s)-v_{t-1}^s},  \label{eq:cauchy}
\end{align}
where (\ref{eq:beta}) uses the following Young's inequality (choose $\alpha =3$)
\begin{equation}\label{eq:young}
 \ns{x_t^s-x_{t-1}^s}\leq \big(1+\frac{1}{\alpha}\big)\ns{\bx_t^s-x_{t-1}^s}
                              +(1+\alpha)\ns{x_t^s-\bx_t^s}, ~~\forall \alpha>0,
\end{equation}
and (\ref{eq:cauchy}) holds due to the following Lemma \ref{lm:2}.

\begin{lemma}\label{lm:2}
Let $x_t^s := \prox_{\eta h}(x_{t-1}^s-\eta v_{t-1}^s)$ and $\bx_t^s := \prox_{\eta h}\big(x_{t-1}^s-\eta \nabla f(x_{t-1}^s)\big)$. Then, the following inequality holds:
\begin{equation*}
\inner{\nabla f(x_{t-1}^s)-v_{t-1}^s}{x_t^s-\bx_t^s} \leq \eta\ns{\nabla f(x_{t-1}^s)-v_{t-1}^s}
\end{equation*}
\end{lemma}
\begin{proofof}{Lemma \ref{lm:2}}
First, we obtain the relation between $\n{x_t^s-\bx_t^s}$ and $\n{\nabla f(x_{t-1}^s)- v_{t-1}^s}$ as follows (similar to \citep{ghadimi2016mini}):
\begin{align}
 h(x_t^s) &\leq h(\bx_t^s) -
    \inner{v_{t-1}^s+\frac{1}{\eta}(x_t^s-x_{t-1}^s)}{x_t^s-\bx_t^s} \label{eq:hc1}\\
 h(\bx_t^s) &\leq h(x_t^s) -
    \inner{\nabla f(x_{t-1}^s)+\frac{1}{\eta}(\bx_t^s-x_{t-1}^s)}{\bx_t^s-x_t^s}, \label{eq:hc2}
\end{align}
where (\ref{eq:hc1}) and (\ref{eq:hc2}) hold due to (\ref{eq:h2}).
Adding (\ref{eq:hc1}) and (\ref{eq:hc2}), we have
\begin{align}
  \frac{1}{\eta}\inner{x_t^s-\bx_t^s}{x_t^s-\bx_t^s}
          &\leq  \inner{\nabla f(x_{t-1}^s)- v_{t-1}^s}{x_t^s-\bx_t^s}  \notag\\
   \frac{1}{\eta}\ns{x_t^s-\bx_t^s}
          &\leq  \n{\nabla f(x_{t-1}^s)- v_{t-1}^s}\n{x_t^s-\bx_t^s} \label{eq:hc4}\\
   \n{x_t^s-\bx_t^s} &\leq \eta\n{\nabla f(x_{t-1}^s)- v_{t-1}^s}, \label{eq:hc5}
\end{align}
where (\ref{eq:hc4}) uses Cauchy-Schwarz inequality.

Now, this lemma is proved by using Cauchy-Schwarz inequality and (\ref{eq:hc5}), i.e.,
$\inner{\nabla f(x_{t-1}^s)-v_{t-1}^s}{x_t^s-\bx_t^s}
\leq \n{\nabla f(x_{t-1}^s)- v_{t-1}^s}\n{x_t^s-\bx_t^s}
\leq \eta\ns{\nabla f(x_{t-1}^s)-v_{t-1}^s}.$
\end{proofof}

Note that $x_t^s = \prox_{\eta h}(x_{t-1}^s-\eta v_{t-1}^s)$ is the iterated form in our algorithm (see Line 7 in Algorithm \ref{alg:1}).
Now, we take expectations with all history for (\ref{eq:cauchy}).
\begin{equation}\label{eq:exp}
\E[\Phi(x_t^s)] \leq \E\Big[\Phi(x_{t-1}^s) -\Big(\frac{5}{8\eta}-\frac{L}{2}\Big)\ns{x_t^s-x_{t-1}^s} -\Big(\frac{1}{3\eta}-L\Big)\ns{\bx_t^s-x_{t-1}^s}
   + \eta\ns{\nabla f(x_{t-1}^s)-v_{t-1}^s}\Big]
\end{equation}

Then, we bound the variance term in (\ref{eq:exp}) as follows:
\begin{align}
 &\E\Big[\eta\ns{\nabla f(x_{t-1}^s)-v_{t-1}^s}\Big] \notag\\
 &= \E\Big[\eta\Big\|\frac{1}{b}\sum_{i\in I_b}\Big(\nabla f_i(x_{t-1}^s)-\nabla f_i(\tx^{s-1})\Big)
            -\big(\nabla f(x_{t-1}^s)-g^s\big)\Big\|^2\Big] \notag\\
 &=\E\Big[\eta\Big\|\frac{1}{b}\sum_{i\in I_b}\Big(\nabla f_i(x_{t-1}^s)-\nabla f_i(\tx^{s-1})\Big)
            -\Big(\nabla f(x_{t-1}^s)
            - \frac{1}{B}\sum_{j\in I_B}\nabla f_j(\tx^{s-1})\Big)\Big\|^2\Big] \notag\\
 &= \E\Big[\eta\Big\|\frac{1}{b}\sum_{i\in I_b}\Big(\nabla f_i(x_{t-1}^s)-\nabla f_i(\tx^{s-1})\Big)
            -\Big(\nabla f(x_{t-1}^s)-\nabla f(\tx^{s-1})\Big)
            +\Big(\frac{1}{B}\sum_{j\in I_B}\nabla f_j(\tx^{s-1})
            - \nabla f(\tx^{s-1})\Big)\Big\|^2\Big] \notag \\
 &= \eta\E\Big[\Big\|\frac{1}{b}
     \sum_{i\in I_b}\Big(\big(\nabla f_i(x_{t-1}^s)-\nabla f_i(\tx^{s-1})\big)
            -\big(\nabla f(x_{t-1}^s)-\nabla f(\tx^{s-1})\big)\Big)
     +\frac{1}{B}\sum_{j\in I_B}\Big(\nabla f_j(\tx^{s-1})- \nabla f(\tx^{s-1})\Big)
     \Big\|^2\Big] \notag \\
 &= \eta\E\Big[\Big\|\frac{1}{b}
     \sum_{i\in I_b}\Big(\big(\nabla f_i(x_{t-1}^s)-\nabla f_i(\tx^{s-1})\big)
            -\big(\nabla f(x_{t-1}^s)-\nabla f(\tx^{s-1})\big)\Big)\Big\|^2\Big] \notag\\
      &\qquad \qquad      +\eta\E\Big[\Big\|
     \frac{1}{B}\sum_{j\in I_B}\Big(\nabla f_j(\tx^{s-1})- \nabla f(\tx^{s-1})\Big)
     \Big\|^2\Big] \label{eq:ijind}\\
 &= \frac{\eta}{b^2}\E\Big[\sum_{i\in I_b}\Big\|
     \Big(\big(\nabla f_i(x_{t-1}^s)- \nabla f_i(\tx^{s-1})\big)
            -\big(\nabla f(x_{t-1}^s)-\nabla f(\tx^{s-1})\big)\Big)\Big\|^2\Big]\notag\\
      &\qquad \qquad      +\eta\E\Big[\Big\|
     \frac{1}{B}\sum_{j\in I_B}\Big(\nabla f_j(\tx^{s-1})- \nabla f(\tx^{s-1})\Big)
     \Big\|^2\Big] \label{eq:v1}\\
 &\leq \frac{\eta}{b^2}
         \E\Big[\sum_{i\in I_b}\big\|\nabla f_i(x_{t-1}^s)- \nabla f_i(\tx^{s-1}) \big\|^2\Big]
         +\eta\E\Big[\Big\|
          \frac{1}{B}\sum_{j\in I_B}\Big(\nabla f_j(\tx^{s-1})- \nabla f(\tx^{s-1})\Big)
         \Big\|^2\Big]  \label{eq:v2}\\
 &\leq \frac{\eta L^2}{b}\E[\ns{x_{t-1}^s-\tx^{s-1}}]
        +\frac{I\{B<n\}\eta\sigma^2}{B} \label{eq:useasp},
\end{align}
where the expectations are taking with $I_b$ and $I_B$.
(\ref{eq:ijind}) and (\ref{eq:v1}) hold since $\E[\ns{x_1+x_2+\cdots+x_k}]=\sum_{i=1}^k\E[\ns{x_i}]$ if $x_1, x_2, \ldots, x_k$ are independent and of mean zero (note that $I_b$ and $I_B$ are also independent).
(\ref{eq:v2}) uses the fact that $\E[\ns{x-\E[x]}]\leq \E[\ns{x}]$, for any random variable $x$.
(\ref{eq:useasp}) holds due to (\ref{eq:smooth}) and Assumption \ref{asp:var}.

Now, we plug (\ref{eq:useasp}) into (\ref{eq:exp}) to obtain
\begin{align}
&\E[\Phi(x_t^s)]  \notag\\
&\leq \E\Big[\Phi(x_{t-1}^s) -\Big(\frac{5}{8\eta}-\frac{L}{2}\Big)\ns{x_t^s-x_{t-1}^s}
   -\Big(\frac{1}{3\eta}-L\Big)\ns{\bx_t^s-x_{t-1}^s}
   + \frac{\eta L^2}{b}\ns{x_{t-1}^s-\tx^{s-1}}
   +\frac{I\{B<n\}\eta\sigma^2}{B}\Big] \label{eq:reuse}\\
&=\E\Big[\Phi(x_{t-1}^s)
   -\frac{13L}{4}\ns{x_t^s-x_{t-1}^s} -L\ns{\bx_t^s-x_{t-1}^s}
   + \frac{L}{6b}\ns{x_{t-1}^s-\tx^{s-1}}
   +\frac{I\{B<n\}\eta\sigma^2}{B}\Big] \label{eq:eta}\\
&=\E\Big[\Phi(x_{t-1}^s)
   -\frac{13L}{4}\ns{x_t^s-x_{t-1}^s} -\frac{1}{36L}\ns{\calG_\eta(x_{t-1}^s)}
   + \frac{L}{6b}\ns{x_{t-1}^s-\tx^{s-1}}
   +\frac{I\{B<n\}\eta\sigma^2}{B}\Big] \label{eq:xtog}\\
& \leq\E\Big[\Phi(x_{t-1}^s)
   -\frac{13L}{8t}\ns{x_t^s-\tx^{s-1}} -\frac{1}{36L}\ns{\calG_\eta(x_{t-1}^s)}
   + \Big(\frac{L}{6b}+\frac{13L}{8t-4}\Big)\ns{x_{t-1}^s-\tx^{s-1}}
   +\frac{I\{B<n\}\eta\sigma^2}{B}\Big], \label{eq:useyoung}
\end{align}
where (\ref{eq:eta}) uses $\eta = \frac{1}{6L}$, and (\ref{eq:xtog}) uses the definition of gradient mapping $\calG_\eta(x_{t-1}^s)$ (see (\ref{eq:gradmap})) and recall $\bx_t^s := \prox_{\eta h}\big(x_{t-1}^s-\eta \nabla f(x_{t-1}^s)\big)$. (\ref{eq:useyoung}) uses $\ns{x_t^s-\tx^{s-1}}\leq \big(1+\frac{1}{\alpha}\big)\ns{x_{t-1}^s-\tx^{s-1}}
+(1+\alpha)\ns{x_t^s-x_{t-1}^s}$ by choosing $\alpha=2t-1$.

Now, adding (\ref{eq:useyoung}) for all iterations $1\le t\le m$ in epoch $s$
and recalling that $x_m^s=\tx^{s}$ and $x_{0}^s = \tx^{s-1}$, we get
\begin{align}
 \E[\Phi(\tx^{s})]&\leq \E\Big[\Phi(\tx^{s-1})
   -\sum_{t=1}^{m}\frac{1}{36L}\ns{\calG_\eta(x_{t-1}^s)}
   -\sum_{t=1}^{m}\frac{13L}{8t}\ns{x_t^s-\tx^{s-1}}\notag\\
   &\qquad \qquad  +\sum_{t=1}^{m}\Big(\frac{L}{6b}+\frac{13L}{8t-4}\Big)\ns{x_{t-1}^s-\tx^{s-1}}
   +\sum_{t=1}^{m}\frac{I\{B<n\}\eta\sigma^2}{B}\Big] \notag\\
  &\leq \E\Big[\Phi(\tx^{s-1})
   -\sum_{t=1}^{m}\frac{1}{36L}\ns{\calG_\eta(x_{t-1}^s)}
   -\sum_{t=1}^{m-1}\frac{13L}{8t}\ns{x_t^s-\tx^{s-1}}\notag\\
   &\qquad \qquad  +\sum_{t=2}^{m}\Big(\frac{L}{6b}+\frac{13L}{8t-4}\Big)\ns{x_{t-1}^s-\tx^{s-1}}
   +\sum_{t=1}^{m}\frac{I\{B<n\}\eta\sigma^2}{B}\Big] \label{eq:pos}\\
  &= \E\Big[\Phi(\tx^{s-1})
   -\sum_{t=1}^{m}\frac{1}{36L}\ns{\calG_\eta(x_{t-1}^s)}
   -\sum_{t=1}^{m-1}\Big(\frac{13L}{8t}-\frac{L}{6b}-\frac{13L}{8t+4}\Big)\ns{x_{t}^s-\tx^{s-1}}
   +\sum_{t=1}^{m}\frac{I\{B<n\}\eta\sigma^2}{B}\Big] \notag\\
  &\leq \E\Big[\Phi(\tx^{s-1})
   -\sum_{t=1}^{m}\frac{1}{36L}\ns{\calG_\eta(x_{t-1}^s)}
   -\sum_{t=1}^{m-1}\Big(\frac{L}{2t^2}-\frac{L}{6b}\Big)\ns{x_{t}^s-\tx^{s-1}}
   +\sum_{t=1}^{m}\frac{I\{B<n\}\eta\sigma^2}{B}\Big] \notag\\
   &\leq \E\Big[\Phi(\tx^{s-1})
   -\sum_{t=1}^{m}\frac{1}{36L}\ns{\calG_\eta(x_{t-1}^s)}
   +\sum_{t=1}^{m}\frac{I\{B<n\}\eta\sigma^2}{B}\Big], \label{eq:final}
\end{align}
where (\ref{eq:pos}) holds since $\ns{\cdot}$ always be non-negative and $x_{0}^s = \tx^{s-1}$,
and (\ref{eq:final}) holds since $m=\sqrt{b}$.
Thus, $\frac{L}{2t^2}-\frac{L}{6b}\geq 0$ for all $1\leq t <m$.

Now, we sum up (\ref{eq:final}) for all epochs $1\le s\le S$ to finish the proof as follows:
\begin{align}
0\leq \E[\Phi(\tx^{S})-\Phi(x^*)] &\leq \E\Big[\Phi(\tx^{0})-\Phi(x^*)
       -\sum_{s=1}^{S}\sum_{t=1}^{m}\frac{1}{36L}\ns{\calG_\eta(x_{t-1}^s)}
       +\sum_{s=1}^{S}\sum_{t=1}^{m}\frac{I\{B<n\}\eta\sigma^2}{B}\Big] \notag\\
\E[\ns{\calG_\eta(\hx)}] &\leq \frac{36L\big(\Phi(x_0)-\Phi(x^*)\big)}{Sm}
       +\frac{I\{B<n\}36L\eta\sigma^2}{B}  \label{eq:plugeta}\\
       &=\frac{36L\big(\Phi(x_0)-\Phi(x^*)\big)}{Sm}
       +\frac{I\{B<n\}6\sigma^2}{B}=2\epsilon, \label{eq:addx}
\end{align}
where (\ref{eq:plugeta}) holds since $\hx$ is chosen uniformly randomly from $\{x_{t-1}^s\}_{t\in[m], s\in[S]}$, and (\ref{eq:addx}) uses $\eta = \frac{1}{6L}$.
Now, we obtain the total number of iterations $T=Sm=S\sqrt{b}=\frac{36L\big(\Phi(x_0)-\Phi(x^*)\big)}{\epsilon}$.
The number of PO calls equals to $T=Sm=\frac{36L\big(\Phi(x_0)-\Phi(x^*)\big)}{\epsilon}$.
The proof is finished since the number of SFO calls equals to $Sn+Smb=36L\big(\Phi(x_0)-\Phi(x^*)\big)\big(\frac{n}{\epsilon\sqrt{b}}+\frac{b}{\epsilon}\big)$
if $B=n$ (i.e., the second term in (\ref{eq:addx}) is 0 and thus Assumption \ref{asp:var} is not needed), or equals to
$SB+Smb=36L\big(\Phi(x_0)-\Phi(x^*)\big)\big(\frac{B}{\epsilon\sqrt{b}}+\frac{b}{\epsilon}\big)$
if $B<n$ (note that $\frac{I\{B<n\}6\sigma^2}{B}\leq \epsilon$ since $B\geq 6\sigma^2/\epsilon$).
\end{proofof}

\subsection{Other Choices of Epoch Length $m$}
\label{app:epoch}
In this section, we show that the similar convergence result (i.e., Theorem $\ref{thm:1}$) holds for other choices of epoch length $m \neq \sqrt{b}$. The difference is that we need to choose different step size $\eta$.
Now, we list the similar convergence result in the following theorem and then prove it.

\begin{theorem}\label{thm:m1}
Let step size $\eta=\min\{\frac{1}{6L}, \frac{\sqrt{b}}{6mL}\}$,
where $b$ is the minibatch size and $m$ is the epoch length.
Then $\hx$ returned by Algorithm \ref{alg:1} is an $\epsilon$-accurate solution for problem (\ref{eq:form}) (i.e., $\E[\ns{\calG_\eta(\hx)}]\leq \epsilon$).
We distinguish the following two cases:
\begin{enumerate}[1)]
  \item We let batch size $B=n$. The number of SFO calls is at most
\begin{equation*}
6\big(\Phi(x_0)-\Phi(x^*)\big)\Big(\frac{n}{\epsilon \eta m}+\frac{b}{\epsilon \eta}\Big).
\end{equation*}
  \item Under Assumption \ref{asp:var}, we let batch size $B=\min\{6\sigma^2/\epsilon,n\}$.
  The number of SFO calls is at most
\begin{equation*}
6\big(\Phi(x_0)-\Phi(x^*)\big)\Big(\frac{B}{\epsilon \eta m}+\frac{b}{\epsilon \eta}\Big).
\end{equation*}
\end{enumerate}
In both cases, the number of PO calls equals to the total number of iterations $T$ which is at most $\frac{6\big(\Phi(x_0)-\Phi(x^*)\big)}{\epsilon \eta}$.
\end{theorem}

\begin{proof}
We recall the Inequality (\ref{eq:reuse}) in the proof of Theorem \ref{thm:1} as follows:
\begin{align}
&\E[\Phi(x_t^s)]  \notag\\
   &\leq \E\Big[\Phi(x_{t-1}^s) -\Big(\frac{5}{8\eta}-\frac{L}{2}\Big)\ns{x_t^s-x_{t-1}^s}
     -\Big(\frac{1}{3\eta}-L\Big)\ns{\bx_t^s-x_{t-1}^s}
     + \frac{\eta L^2}{b}\ns{x_{t-1}^s-\tx^{s-1}}
     +\frac{I\{B<n\}\eta\sigma^2}{B}\Big] \notag\\
   &=\E\Big[\Phi(x_{t-1}^s)
     -\Big(\frac{5}{8\eta}-\frac{L}{2}\Big)\ns{x_t^s-x_{t-1}^s} -\Big(\frac{1}{3\eta}-L\Big)\eta^2\ns{\calG_\eta(x_{t-1}^s)}
     + \frac{\eta L^2}{b}\ns{x_{t-1}^s-\tx^{s-1}}
     +\frac{I\{B<n\}\eta\sigma^2}{B}\Big] \label{eq:mmap}\\
   &\leq\E\Big[\Phi(x_{t-1}^s)
     -\Big(\frac{5}{8\eta}-\frac{L}{2}\Big)\ns{x_t^s-x_{t-1}^s} -\frac{\eta}{6}\ns{\calG_\eta(x_{t-1}^s)}
     + \frac{\eta L^2}{b}\ns{x_{t-1}^s-\tx^{s-1}}
     +\frac{I\{B<n\}\eta\sigma^2}{B}\Big] \label{eq:meta}\\
   & \leq\E\Big[\Phi(x_{t-1}^s)
     -\frac{1}{2t}\Big(\frac{5}{8\eta}-\frac{L}{2}\Big)\ns{x_t^s-\tx^{s-1}} -\frac{\eta}{6}\ns{\calG_\eta(x_{t-1}^s)} \notag\\
     &\qquad \qquad + \Big(\frac{\eta L^2}{b}
      +\frac{1}{2t-1}\big(\frac{5}{8\eta}-\frac{L}{2}\big)\Big)\ns{x_{t-1}^s-\tx^{s-1}}
      +\frac{I\{B<n\}\eta\sigma^2}{B}\Big], \label{eq:last}
\end{align}
where (\ref{eq:mmap}) uses the definition of gradient mapping $\calG_\eta(x_{t-1}^s)$ (see (\ref{eq:gradmap})) and recall $\bx_t^s := \prox_{\eta h}\big(x_{t-1}^s-\eta \nabla f(x_{t-1}^s)\big)$.
(\ref{eq:meta}) uses $\eta \leq \frac{1}{6L}$.
(\ref{eq:last}) uses $\ns{x_t^s-\tx^{s-1}}\leq \big(1+\frac{1}{\alpha}\big)\ns{x_{t-1}^s-\tx^{s-1}}
+(1+\alpha)\ns{x_t^s-x_{t-1}^s}$ by choosing $\alpha=2t-1$.

Now, the remaining proof is almost the same as that of Theorem \ref{thm:1}.
Adding $(\ref{eq:last})$ for all iterations $1\le t\le m$ in epoch $s$
and recalling that $x_m^s=\tx^{s}$ and $x_{0}^s = \tx^{s-1}$,
we have
\begin{align}
 &\E[\Phi(\tx^{s})] \notag\\
  &\leq \E\Big[\Phi(\tx^{s-1})
    -\sum_{t=1}^{m}\frac{\eta}{6}\ns{\calG_\eta(x_{t-1}^s)}
    -\sum_{t=1}^{m}\frac{1}{2t}\Big(\frac{5}{8\eta}-\frac{L}{2}\Big)\ns{x_t^s-\tx^{s-1}} \notag\\
    &\qquad \qquad +\sum_{t=1}^{m}\Big(\frac{\eta L^2}{b}
      +\frac{1}{2t-1}\big(\frac{5}{8\eta}-\frac{L}{2}\big)\Big)
      \ns{x_{t-1}^s-\tx^{s-1}}
      +\sum_{t=1}^{m}\frac{I\{B<n\}\eta\sigma^2}{B}\Big] \notag\\
  &\leq \E\Big[\Phi(\tx^{s-1})
    -\sum_{t=1}^{m}\frac{\eta}{6}\ns{\calG_\eta(x_{t-1}^s)}
    -\sum_{t=1}^{m-1}\frac{1}{2t}\Big(\frac{5}{8\eta}-\frac{L}{2}\Big)\ns{x_t^s-\tx^{s-1}} \notag\\
    &\qquad \qquad +\sum_{t=2}^{m}\Big(\frac{\eta L^2}{b}
      +\frac{1}{2t-1}\big(\frac{5}{8\eta}-\frac{L}{2}\big)\Big)
      \ns{x_{t-1}^s-\tx^{s-1}}
      +\sum_{t=1}^{m}\frac{I\{B<n\}\eta\sigma^2}{B}\Big] \label{eq:mpos}\\
  &= \E\Big[\Phi(\tx^{s-1})
    -\sum_{t=1}^{m}\frac{\eta}{6}\ns{\calG_\eta(x_{t-1}^s)}
    -\sum_{t=1}^{m-1}\Big(\big(\frac{1}{2t}-\frac{1}{2t+1}\big)
    \big(\frac{5}{8\eta}-\frac{L}{2}\big)-\frac{\eta L^2}{b}\Big)\ns{x_t^s-\tx^{s-1}}
    +\sum_{t=1}^{m}\frac{I\{B<n\}\eta\sigma^2}{B}\Big] \notag\\
  &\leq \E\Big[\Phi(\tx^{s-1})
    -\sum_{t=1}^{m}\frac{\eta}{6}\ns{\calG_\eta(x_{t-1}^s)}
    -\sum_{t=1}^{m-1}\Big(\frac{1}{6t^2}
    \big(\frac{5}{8\eta}-\frac{L}{2}\big)-\frac{\eta L^2}{b}\Big)\ns{x_t^s-\tx^{s-1}}
    +\sum_{t=1}^{m}\frac{I\{B<n\}\eta\sigma^2}{B}\Big] \notag\\
  &\leq \E\Big[\Phi(\tx^{s-1})
    -\sum_{t=1}^{m}\frac{\eta}{6}\ns{\calG_\eta(x_{t-1}^s)}
    +\sum_{t=1}^{m}\frac{I\{B<n\}\eta\sigma^2}{B}\Big], \label{eq:mfinal}
\end{align}
where (\ref{eq:mpos}) holds since $\ns{\cdot}$ always be non-negative and $x_{0}^s = \tx^{s-1}$,
and (\ref{eq:mfinal}) holds since it is sufficient to show that
$\big(\frac{1}{6m^2}\big(\frac{5}{8\eta}-\frac{L}{2}\big)-\frac{\eta L^2}{b}\big)\geq 0$.
This holds since $\eta=\min\{\frac{1}{6L}, \frac{\sqrt{b}}{6mL}\}$.

Now, we sum up (\ref{eq:mfinal}) for all epochs $1\le s\le S$ to finish the proof as follows:
\begin{align}
0\leq \E[\Phi(\tx^{S})-\Phi(x^*)] &\leq \E\Big[\Phi(\tx^{0})-\Phi(x^*)
       -\sum_{s=1}^{S}\sum_{t=1}^{m}\frac{\eta}{6}\ns{\calG_\eta(x_{t-1}^s)}
       +\sum_{s=1}^{S}\sum_{t=1}^{m}\frac{I\{B<n\}\eta\sigma^2}{B}\Big] \notag\\
  \E[\ns{\calG_\eta(\hx)}] &\leq \frac{6\big(\Phi(x_0)-\Phi(x^*)\big)}{\eta Sm}
       +\frac{I\{B<n\}6\sigma^2}{B}=2\epsilon, \label{eq:addx5}
\end{align}
where (\ref{eq:addx5}) holds since $\hx$ is chosen uniformly randomly from $\{x_{t-1}^s\}_{t\in[m], s\in[S]}$.
Now, we obtain the total number of iterations $T=Sm=\frac{6\big(\Phi(x_0)-\Phi(x^*)\big)}{\epsilon \eta}$.
The number of PO calls equals to $T=Sm=\frac{6\big(\Phi(x_0)-\Phi(x^*)\big)}{\epsilon \eta}$.
The proof is finished since the number of SFO calls equals to
$Sn+Smb=6\big(\Phi(x_0)-\Phi(x^*)\big)\big(\frac{n}{\epsilon \eta m}+\frac{b}{\epsilon \eta}\big)$
if $B=n$ (i.e., the second term in (\ref{eq:addx5}) is 0 and thus Assumption \ref{asp:var} is not needed), or equals to
$SB+Smb=6\big(\Phi(x_0)-\Phi(x^*)\big)\Big(\frac{B}{\epsilon \eta m}+\frac{b}{\epsilon \eta}\Big)$
if $B<n$ (note that $\frac{I\{B<n\}6\sigma^2}{B}\leq \epsilon$ since $B\geq 6\sigma^2/\epsilon$).
\end{proof}

\section{Proof for ProxSVRG+ Under PL Condition}

In this appendix, we first provide the proof of Theorem \ref{thm:pl1} under the PL condition with form (\ref{eq:gelpl}) (Appendix \ref{app:pl}).
Then we also provide the proof of Theorem \ref{thm:pl1} under the PL condition with form (\ref{eq:plreddi}) (Appendix \ref{app:pl2}).

\subsection{Proof Under PL Form (\ref{eq:gelpl})}
\label{app:pl}

\begin{proofof}{Theorem \ref{thm:pl1}}
First, we recall a key inequality (\ref{eq:useyoung}) from the proof of Theorem \ref{thm:1}, i.e.,
\begin{align}
&\E[\Phi(x_t^s)] \notag\\
   & \leq\E\Big[\Phi(x_{t-1}^s)
   -\frac{13L}{8t}\ns{x_t^s-\tx^{s-1}} -\frac{1}{36L}\ns{\calG_\eta(x_{t-1}^s)}
   + \Big(\frac{L}{6b}+\frac{13L}{8t-4}\Big)\ns{x_{t-1}^s-\tx^{s-1}}
   +\frac{I\{B<n\}\eta\sigma^2}{B}\Big]. \label{eq:plfirst}
\end{align}

Then, we plug the following PL inequality (see (\ref{eq:gelpl}))
\begin{equation*}
\ns{\calG_\eta(x_{t-1}^s)} \geq 2\mu (\Phi(x_{t-1}^s)-\Phi^*)
\end{equation*}
into (\ref{eq:plfirst}) to get
\begin{align}
&\E[\Phi(x_t^s)] \notag\\
   & \leq\E\Big[\Phi(x_{t-1}^s)
   -\frac{13L}{8t}\ns{x_t^s-\tx^{s-1}} -\frac{\mu}{18L}(\Phi(x_{t-1}^s)-\Phi^*)
   + \Big(\frac{L}{6b}+\frac{13L}{8t-4}\Big)\ns{x_{t-1}^s-\tx^{s-1}}
   +\frac{I\{B<n\}\eta\sigma^2}{B}\Big].  \notag\\
\end{align}
Then, we obtain
\begin{align}
&\E[\Phi(x_t^s)-\Phi^*] \notag\\
   & \leq\E\Big[\Big(1-\frac{\mu}{18L}\Big)\big(\Phi(x_{t-1}^s)-\Phi^*\big)
   -\frac{13L}{8t}\ns{x_t^s-\tx^{s-1}}
   + \Big(\frac{L}{6b}+\frac{13L}{8t-4}\Big)\ns{x_{t-1}^s-\tx^{s-1}}
   +\frac{I\{B<n\}\eta\sigma^2}{B}\Big].  \label{eq:plit}
\end{align}
Let $\alpha := 1-\frac{\mu}{18L}$ and $\Psi_t^s := \frac{\E[\Phi(x_t^s)-\Phi^*]}{\alpha^t}$.
Plugging them into (\ref{eq:plit}), we have
\begin{align}
\Psi_t^s
   & \leq \Psi_{t-1}^s
   -\E\Big[\frac{13L}{8t\alpha^t}\ns{x_t^s-\tx^{s-1}}
    -\frac{1}{\alpha^t}\Big(\frac{L}{6b}+\frac{13L}{8t-4}\Big)\ns{x_{t-1}^s-\tx^{s-1}}
    -\frac{1}{\alpha^t}\frac{I\{B<n\}\eta\sigma^2}{B}\Big]. \label{eq:pl5}
\end{align}
Now, adding (\ref{eq:pl5}) from all iterations $1\leq t \leq m$ in epoch $s$
and recalling that $x_m^s=\tx^{s}$ and $x_{0}^s = \tx^{s-1}$, we have
\begin{align}
&\E[\Phi(\tx^s)-\Phi^*] \notag\\
   & \leq \alpha^m\E[\Phi(\tx^{s-1})-\Phi^*]
     +\alpha^m\sum_{t=1}^{m}\frac{1}{\alpha^t}\frac{I\{B<n\}\eta\sigma^2}{B}\notag\\
     &\qquad \qquad -\alpha^m\E\Big[\sum_{t=1}^{m}\frac{13L}{8t\alpha^t}\ns{x_t^s-\tx^{s-1}}
      -\sum_{t=1}^{m}\frac{1}{\alpha^t}\Big(\frac{L}{6b}+\frac{13L}{8t-4}\Big)
      \ns{x_{t-1}^s-\tx^{s-1}}\Big] \notag \\
   & = \alpha^m\E[\Phi(\tx^{s-1})-\Phi^*]
     +\frac{1-\alpha^m}{1-\alpha}\frac{I\{B<n\}\eta\sigma^2}{B}\notag\\
     &\qquad \qquad -\alpha^m\E\Big[\sum_{t=1}^{m}\frac{13L}{8t\alpha^t}\ns{x_t^s-\tx^{s-1}}
      -\sum_{t=1}^{m}\frac{1}{\alpha^t}\Big(\frac{L}{6b}+\frac{13L}{8t-4}\Big)
      \ns{x_{t-1}^s-\tx^{s-1}}\Big] \notag \\
   & \leq \alpha^m\E[\Phi(\tx^{s-1})-\Phi^*]
     +\frac{1-\alpha^m}{1-\alpha}\frac{I\{B<n\}\eta\sigma^2}{B}\notag\\
     &\qquad \qquad -\alpha^m\E\Big[\sum_{t=1}^{m-1}\frac{13L}{8t\alpha^t}\ns{x_t^s-\tx^{s-1}}
       -\sum_{t=2}^{m}\frac{1}{\alpha^t}\Big(\frac{L}{6b}+\frac{13L}{8t-4}\Big)
       \ns{x_{t-1}^s-\tx^{s-1}}\Big] \label{eq:plt1}\\
   & = \alpha^m\E[\Phi(\tx^{s-1})-\Phi^*]
       +\frac{1-\alpha^m}{1-\alpha}\frac{I\{B<n\}\eta\sigma^2}{B}  \notag\\
       &\qquad \qquad -\alpha^m\E\Big[\sum_{t=1}^{m-1}\frac{1}{\alpha^{t+1}}
       \Big(\frac{13L\alpha}{8t}-\frac{L}{6b}-\frac{13L}{8t+4}\Big)
        \ns{x_t^s-\tx^{s-1}}\Big] \notag\\
   & \leq \alpha^m\E[\Phi(\tx^{s-1})-\Phi^*]
       +\frac{1-\alpha^m}{1-\alpha}\frac{I\{B<n\}\eta\sigma^2}{B}\notag\\
       &\qquad \qquad -\alpha^m\E\Big[\sum_{t=1}^{m-1}\frac{1}{\alpha^{t+1}}
        \Big(\frac{13L}{8t}\big(1-\frac{1}{18\sqrt{n}}\big)-\frac{L}{6b}-\frac{13L}{8t+4}\Big)
         \ns{x_t^s-\tx^{s-1}}\Big] \label{eq:plt3}\\
   & \leq \alpha^m\E[\Phi(\tx^{s-1})-\Phi^*]
        +\frac{1-\alpha^m}{1-\alpha}\frac{I\{B<n\}\eta\sigma^2}{B}
         -\alpha^m\E\Big[\sum_{t=1}^{m-1}\frac{L}{\alpha^{t+1}}
          \Big(\frac{1}{2t^2}-\frac{13}{8t}\frac{1}{18\sqrt{n}}-\frac{1}{6b}\Big)
          \ns{x_t^s-\tx^{s-1}}\Big] \notag\\
   & \leq \alpha^m\E[\Phi(\tx^{s-1})-\Phi^*]
       +\frac{1-\alpha^m}{1-\alpha}\frac{I\{B<n\}\eta\sigma^2}{B}
       -\alpha^m\E\Big[\sum_{t=1}^{m-1}\frac{L}{\alpha^{t+1}}
       \Big(\frac{1}{2t^2}-\frac{1}{8\sqrt{n}t}-\frac{1}{6b}\Big)
       \ns{x_t^s-\tx^{s-1}}\Big] \notag\\
    & \leq \alpha^m\E[\Phi(\tx^{s-1})-\Phi^*]
       +\frac{1-\alpha^m}{1-\alpha}\frac{I\{B<n\}\eta\sigma^2}{B}, \label{eq:pllast}
\end{align}
where (\ref{eq:plt1}) holds since $\ns{\cdot}$ always be non-negative and $x_{0}^s = \tx^{s-1}$.
(\ref{eq:plt3}) holds since $\alpha = 1-\frac{\mu}{18L}$ and the assumption $L/\mu > \sqrt{n}$.
(\ref{eq:pllast}) holds since it is sufficient to show that
$\Gamma_t\geq 0$ for all $1\le t<m$, where $\Gamma_t=\frac{1}{2t^2}-\frac{1}{8\sqrt{n}t}-\frac{1}{6b}$.
Taking a derivative for $\Gamma_t$, we get
$\Gamma_t'=-\frac{1}{t^3}+\frac{1}{8\sqrt{n}t^2}=-\frac{8\sqrt{n}-t}{8\sqrt{n}t^3}<0$ since $t<m=\sqrt{b}\le\sqrt{n}$ (note that for other choices of epoch length $m$, the proof is almost the same as that in Appendix \ref{app:epoch}).
Thus, $\Gamma_t$ decreases in $t$.
We only need to show that $\Gamma_m=\Gamma_{\sqrt{b}}\geq 0$, i.e.,
$\frac{1}{2b}-\frac{1}{8\sqrt{nb}}-\frac{1}{6b} = \frac{1}{3b} - \frac{1}{8\sqrt{nb}}\geq 0$.
It is easy to see that this inequality holds since $b\leq n$.

Similarly, let $\widetilde{\alpha}:=\alpha^m$ and $\widetilde{\Psi}^s := \frac{\E[\Phi(\tx^s)-\Phi^*]}{\widetilde{\alpha}^s}$.
Plugging them into (\ref{eq:pllast}), we have
\begin{align}
\widetilde{\Psi}^s
   & \leq \widetilde{\Psi}^{s-1}
    -\frac{1}{\widetilde{\alpha}^s}\frac{1-\widetilde{\alpha}}{1-\alpha}\frac{I\{B<n\}\eta\sigma^2}{B}. \label{eq:pllast1}
\end{align}
Now, we sum up (\ref{eq:pllast1}) for all epochs $1\le s\le S$ to finish the proof as follows:
\begin{align}
\E[\Phi(\tx^S)-\Phi^*] &\leq \widetilde{\alpha}^S\E[\Phi(\tx^{0})-\Phi^*]
       +\widetilde{\alpha}^S\sum_{s=1}^{S}\frac{1}{\widetilde{\alpha}^s}
         \frac{1-\widetilde{\alpha}}{1-\alpha}\frac{I\{B<n\}\eta\sigma^2}{B} \notag\\
   &=\alpha^{Sm}\E[\Phi(\tx^{0})-\Phi^*]
     +\frac{1-\widetilde{\alpha}^S}{1-\widetilde{\alpha}}\frac{1-\widetilde{\alpha}}{1-\alpha}
     \frac{I\{B<n\}\eta\sigma^2}{B} \notag\\
   &\leq \alpha^{Sm}\E[\Phi(\tx^{0})-\Phi^*]
     +\frac{1}{1-\alpha}\frac{I\{B<n\}\eta\sigma^2}{B} \notag\\
   &=\Big(1-\frac{\mu}{18L}\Big)^{Sm}\big(\Phi(x_0)-\Phi^*\big)
      +\frac{I\{B<n\}18L\eta\sigma^2}{\mu B} \label{eq:plpluga}\\
   &=\Big(1-\frac{\mu}{18L}\Big)^{Sm}\big(\Phi(x_0)-\Phi^*\big)
      +\frac{I\{B<n\}3\sigma^2}{\mu B}
      =2\epsilon, \label{eq:plsm}
\end{align}
where (\ref{eq:plpluga}) holds sine $\alpha=1-\frac{\mu}{18L}$, and (\ref{eq:plsm}) uses $\eta = \frac{1}{6L}$.

From (\ref{eq:plsm}), we obtain the total number of iterations $T=Sm=S\sqrt{b}=O(\frac{1}{\mu}\log\frac{1}{\epsilon})$.
The number of PO calls equals to $T=Sm=O(\frac{1}{\mu}\log\frac{1}{\epsilon})$.
The number of SFO calls equals to $Sn+Smb=O\big(\frac{n}{\mu\sqrt{b}}\log\frac{1}{\epsilon}
+\frac{b}{\mu}\log\frac{1}{\epsilon}\big)$
if $B=n$ (i.e., the second term in (\ref{eq:plsm}) is 0 and thus Assumption \ref{asp:var} is not needed), or equals to
$SB+Smb=O\big(\frac{B}{\mu\sqrt{b}}\log\frac{1}{\epsilon}
+\frac{b}{\mu}\log\frac{1}{\epsilon}\big)$
if $B<n$ (note that $\frac{I\{B<n\}3\sigma^2}{\mu B}\leq \epsilon$ since $B\geq 6\sigma^2/\mu\epsilon$).
\end{proofof}

\subsection{Proof Under Form (\ref{eq:plreddi})}
\label{app:pl2}

\begin{proofof}{Theorem \ref{thm:pl1}}
First, similar to \citep{reddi2016proximal}, we need the following inequality:
\begin{align}
\Phi(\bx_t^s) &= f(\bx_t^s) + h(\bx_t^s) + h(x_{t-1}^s) -h(x_{t-1}^s) \notag\\
        &\leq f(x_{t-1}^s)+ \inner{\nabla f(x_{t-1}^s)}{\bx_t^s - x_{t-1}^s}
        +\frac{L}{2}\ns{\bx_t^s - x_{t-1}^s} + h(\bx_t^s) + h(x_{t-1}^s) -h(x_{t-1}^s) \label{eq:pl2:1}\\
        &= \Phi(x_{t-1}^s)+ \inner{\nabla f(x_{t-1}^s)}{\bx_t^s - x_{t-1}^s}
        +\frac{L}{2}\ns{\bx_t^s - x_{t-1}^s} + h(\bx_t^s) -h(x_{t-1}^s) \notag\\
        &\leq \Phi(x_{t-1}^s)+ \inner{\nabla f(x_{t-1}^s)}{\bx_t^s - x_{t-1}^s}
        +\frac{1}{2\eta}\ns{\bx_t^s - x_{t-1}^s} + h(\bx_t^s) -h(x_{t-1}^s) \label{eq:pl2:2}\\
        &= \Phi(x_{t-1}^s) - \frac{\eta}{2}D_h(x_{t-1}^s,\frac{1}{\eta})  \label{eq:pl2:3}\\
        &\leq \Phi(x_{t-1}^s) - \eta\mu(\Phi(x_{t-1}^s)-\Phi^*),   \label{eq:pl2:4}
\end{align}
where (\ref{eq:pl2:1}) holds since $f$ has $L$-Lipschitz continuous gradient,
(\ref{eq:pl2:2}) holds due to $\eta=\frac{1}{6L}<\frac{1}{L}$,
(\ref{eq:pl2:3}) follows from the definition of $D_h$ and recall $\bx_t^s := \prox_{\eta h}\big(x_{t-1}^s-\eta \nabla f(x_{t-1}^s)\big)$,
and (\ref{eq:pl2:4}) follows from the definition of PL condition with form (\ref{eq:plreddi}).

Then, adding $\frac{9}{11}$ times (\ref{eq:4}) and $\frac{2}{11}$ times (\ref{eq:pl2:4}), we have
\begin{align}
\Phi(\bx_t^s) &\leq  \Phi(x_{t-1}^s)
                  -\frac{9}{11}\Big(\frac{1}{\eta}-\frac{L}{2}\Big)\ns{\bx_t^s-x_{t-1}^s}
                  -\frac{2}{11}\eta\mu(\Phi(x_{t-1}^s)-\Phi^*) \notag\\
             &=\Phi(x_{t-1}^s)
                  -\Big(\frac{9}{11\eta}-\frac{9L}{22}\Big)\ns{\bx_t^s-x_{t-1}^s}
                  -\frac{2\eta\mu}{11}(\Phi(x_{t-1}^s)-\Phi^*).   \label{eq:pl2:5}
\end{align}
We add (\ref{eq:pl2:5}) and (\ref{eq:3}) to obtain the following inequality:
\begin{align}
  \Phi(x_t^s) &\leq \Phi(x_{t-1}^s) + \frac{L}{2}\ns{x_t^s-x_{t-1}^s}
              -\Big(\frac{9}{11\eta}-\frac{9L}{22}-\frac{L}{2}\Big)\ns{\bx_t^s-x_{t-1}^s}
              -\frac{2\eta\mu}{11}(\Phi(x_{t-1}^s)-\Phi^*) \notag \\
           &\qquad \quad -\frac{1}{\eta}\inner{x_t^s-x_{t-1}^s}{x_t^s-\bx_t^s}
                   + \inner{\nabla f(x_{t-1}^s)-v_{t-1}^s}{x_t^s-\bx_t^s} \notag\\
  &= \Phi(x_{t-1}^s) + \frac{L}{2}\ns{x_t^s-x_{t-1}^s}
         -\Big(\frac{9}{11\eta}-\frac{9L}{22}-\frac{L}{2}\Big)\ns{\bx_t^s-x_{t-1}^s}
         -\frac{2\eta\mu}{11}(\Phi(x_{t-1}^s)-\Phi^*) \notag \\
     &\qquad \quad
        -\frac{1}{2\eta}\big(\ns{x_t^s-x_{t-1}^s}+\ns{x_t^s-\bx_t^s}-\ns{\bx_t^s-x_{t-1}^s}\big)
        + \inner{\nabla f(x_{t-1}^s)-v_{t-1}^s}{x_t^s-\bx_t^s} \notag\\
  &= \Phi(x_{t-1}^s) -\Big(\frac{1}{2\eta}-\frac{L}{2}\Big)\ns{x_t^s-x_{t-1}^s}
          -\Big(\frac{7}{22\eta}-\frac{10L}{11}\Big)\ns{\bx_t^s-x_{t-1}^s}
          -\frac{2\eta\mu}{11}(\Phi(x_{t-1}^s)-\Phi^*) \notag \\
       &\qquad \quad -\frac{1}{2\eta}\ns{x_t^s-\bx_t^s}
           + \inner{\nabla f(x_{t-1}^s)-v_{t-1}^s}{x_t^s-\bx_t^s} \notag\\
  &\leq \Phi(x_{t-1}^s) -\Big(\frac{1}{2\eta}-\frac{L}{2}\Big)\ns{x_t^s-x_{t-1}^s}
          -\Big(\frac{7}{22\eta}-\frac{10L}{11}\Big)\ns{\bx_t^s-x_{t-1}^s}
          -\frac{2\eta\mu}{11}(\Phi(x_{t-1}^s)-\Phi^*) \notag \\
       &\qquad \quad -\frac{1}{8\eta}\ns{x_t^s-x_{t-1}^s}+\frac{1}{6\eta}\ns{\bx_t^s-x_{t-1}^s}
           + \inner{\nabla f(x_{t-1}^s)-v_{t-1}^s}{x_t^s-\bx_t^s} \label{eq:pl2:beta}\\
  &= \Phi(x_{t-1}^s) -\Big(\frac{5}{8\eta}-\frac{L}{2}\Big)\ns{x_t^s-x_{t-1}^s}
          -\Big(\frac{5}{33\eta}-\frac{10L}{11}\Big)\ns{\bx_t^s-x_{t-1}^s}
          -\frac{2\eta\mu}{11}(\Phi(x_{t-1}^s)-\Phi^*) \notag \\
       &\qquad \quad
           + \inner{\nabla f(x_{t-1}^s)-v_{t-1}^s}{x_t^s-\bx_t^s} \notag\\
  &\leq \Phi(x_{t-1}^s) -\Big(\frac{5}{8\eta}-\frac{L}{2}\Big)\ns{x_t^s-x_{t-1}^s}
          -\Big(\frac{5}{33\eta}-\frac{10L}{11}\Big)\ns{\bx_t^s-x_{t-1}^s}
          -\frac{2\eta\mu}{11}(\Phi(x_{t-1}^s)-\Phi^*) \notag \\
       &\qquad \quad
           + \eta\ns{\nabla f(x_{t-1}^s)-v_{t-1}^s}.  \label{eq:pl2:cauchy}
\end{align}
In the same way as (\ref{eq:beta}) and (\ref{eq:cauchy}), (\ref{eq:pl2:beta}) uses Young's inequality (\ref{eq:young}) (choose $\alpha =3$)
and (\ref{eq:pl2:cauchy}) follows from Lemma \ref{lm:2}.

Now, we take expectations for (\ref{eq:pl2:cauchy}) and then plug the variance bound (\ref{eq:useasp}) into it to obtain
\begin{align}
&\E[\Phi(x_t^s)]  \notag\\
&\leq\E\Big[\Phi(x_{t-1}^s) -\Big(\frac{5}{8\eta}-\frac{L}{2}\Big)\ns{x_t^s-x_{t-1}^s}
          -\Big(\frac{5}{33\eta}-\frac{10L}{11}\Big)\ns{\bx_t^s-x_{t-1}^s}
          -\frac{2\eta\mu}{11}(\Phi(x_{t-1}^s)-\Phi^*) \notag \\
       &\qquad \quad
         + \frac{\eta L^2}{b}\ns{x_{t-1}^s-\tx^{s-1}}
         +\frac{I\{B<n\}\eta\sigma^2}{B}\Big] \notag\\
&=\E\Big[\Phi(x_{t-1}^s)
   -\frac{13L}{4}\ns{x_t^s-x_{t-1}^s} -\frac{\mu}{33L}(\Phi(x_{t-1}^s)-\Phi^*)
   + \frac{L}{6b}\ns{x_{t-1}^s-\tx^{s-1}}
   +\frac{I\{B<n\}\eta\sigma^2}{B}\Big] \label{eq:pl2:eta}\\
& \leq\E\Big[\Phi(x_{t-1}^s)
   -\frac{13L}{8t}\ns{x_t^s-\tx^{s-1}} -\frac{\mu}{33L}(\Phi(x_{t-1}^s)-\Phi^*)
   + \Big(\frac{L}{6b}+\frac{13L}{8t-4}\Big)\ns{x_{t-1}^s-\tx^{s-1}}
   +\frac{I\{B<n\}\eta\sigma^2}{B}\Big], \label{eq:pl2:useyoung}
\end{align}
where \eqref{eq:pl2:eta} uses $\eta = \frac{1}{6L}$, and (\ref{eq:pl2:useyoung}) uses Young's inequality $\ns{x_t^s-\tx^{s-1}}\leq \big(1+\frac{1}{\alpha}\big)\ns{x_{t-1}^s-\tx^{s-1}}
+(1+\alpha)\ns{x_t^s-x_{t-1}^s}$ by choosing $\alpha=2t-1$.

Now, according to (\ref{eq:pl2:useyoung}), we obtain the following key inequality
\begin{align}
&\E[\Phi(x_t^s)-\Phi^*] \notag\\
   & \leq\E\Big[\Big(1-\frac{\mu}{33L}\Big)\big(\Phi(x_{t-1}^s)-\Phi^*\big)
   -\frac{13L}{8t}\ns{x_t^s-\tx^{s-1}}
   + \Big(\frac{L}{6b}+\frac{13L}{8t-4}\Big)\ns{x_{t-1}^s-\tx^{s-1}}
   +\frac{I\{B<n\}\eta\sigma^2}{B}\Big].  \label{eq:pl2:plit}
\end{align}
The remaining proof is exactly the same as our proof in Appendix \ref{app:pl} from (\ref{eq:plit}) to the end.
\end{proofof}

\end{document}